\documentclass[11pt]{amsart}
\usepackage[T1]{fontenc}
\usepackage[text={6in,8.5in},centering]{geometry}
\usepackage[colorlinks=true,citecolor=black!60!green,linkcolor=black!60!red,filecolor=black!60!cyan,urlcolor=black!60!magenta]{hyperref}
\usepackage{xcolor}
\usepackage{amssymb,amsmath,amsthm,mathtools,extpfeil}
\usepackage[all,cmtip]{xy}
\usepackage{tikz,tikz-3dplot}
\usepackage[shortlabels]{enumitem}

\usepackage{caption}
\usepackage{subcaption}

\newtheorem{thm}[equation]{Theorem}
\newtheorem*{thm*}{Theorem}
\newtheorem{thmA}{Theorem}

\newtheorem{lem}[equation]{Lemma}
\newtheorem{prop}[equation]{Proposition}
\newtheorem{cor}[equation]{Corollary}

\theoremstyle{definition}

\newtheorem*{rmk}{Remark}

\newtheorem*{qn}{Question}

\newtheorem{ex}[equation]{Example}

\numberwithin{equation}{section}

\DeclareMathOperator{\MC}{MC}
\DeclareMathOperator{\rk}{rk}
\DeclareMathOperator{\Sign}{sign}

\DeclareMathOperator{\ch}{ch}
\DeclareMathOperator{\ev}{ev}
\DeclareMathOperator{\U}{U}
\DeclareMathOperator{\SU}{SU}
\DeclareMathOperator{\SO}{SO}
\DeclareMathOperator{\SL}{SL}
\DeclareMathOperator{\GL}{GL}
\DeclareMathOperator{\Sp}{Sp}
\DeclareMathOperator{\Spin}{Spin}
\DeclareMathOperator{\Sym}{Sym}

\newcommand{\ZZ}{\mathbb{Z}}
\newcommand{\QQ}{\mathbb{Q}}
\newcommand{\RR}{\mathbb{R}}
\newcommand{\CC}{\mathbb{C}}

\newcommand*\strong[1]{\textbf{\textit{#1}}}

\makeatletter
\let\@wraptoccontribs\wraptoccontribs
\makeatother

\begin{document}
\title{Morse complexity of homology classes}
\author[F.~Manin]{Fedor Manin}
\author[B.~Tshishiku]{Bena Tshishiku}
\author[Sh.~Weinberger]{Shmuel Weinberger}
\address[F.~Manin]{Department of Mathematics, University of Toronto, ON, Canada}
\email{manin@math.toronto.edu}
\address[B.~Tshishiku]{Department of Mathematics, Brown University, Providence, RI, United States}
\email{bena\_tshishiku@brown.edu}
\address[Sh.~Weinberger]{Department of Mathematics, University of Chicago, IL, United States}
\email{shmuel@math.uchicago.edu}
\begin{abstract}
  The \emph{Morse complexity} of a manifold is the minimal number of handles
  required to build it.  We explore the Morse complexity of manifolds,
  bordisms, and homology classes, proving nontrivial upper bounds using
  surgery theory and lower bounds using index theory.  Our most involved result
  shows that for Lie groups which admit discrete series representations, the
  Morse complexity of their locally symmetric spaces grows linearly with
  volume.  This implies that such locally symmetric spaces do not admit open
  book decompositions.
\end{abstract}
\maketitle

\section{Introduction}

One way to measure the complexity of a smooth manifold $M$ is by the minimal number of critical points of a Morse function on $M$; equivalently, the number of handles needed to build $M$, starting from its boundary if it has one. We call this the Morse complexity of $M$, and denote it by $\MC(M)$.  Note that this is not a conventional complexity measure in that infinitely many manifolds can have the same Morse complexity.  This means, for example, that whether it grows without bound for a given class of manifolds can be an interesting question.

Gromov suggested studying the analogue of this invariant for homology classes, defined as follows.  If the homology class $\alpha \in H_i(X; \ZZ)$ is represented by a manifold mapping in, i.e.\ there is an $f:M \to X$ with $f_*([M]) = \alpha$, then we will say that $\MC(\alpha) \leq \MC(M)$.  From the work of Thom, we know that not every $\alpha$ is of this form, but that $k\alpha$ always is for some positive $k$. So, working rationally, we set
\[\MC(\alpha) = \inf \{\MC(M)/k \mid \text{there is an $f:M \to X$ with }f_*([M]) = k\alpha \}.\]
Note that even if the original $\alpha$ was the fundamental class of a manifold cycle, the Morse complexity of $\alpha$ may well be smaller, because multiples might simplify.  Indeed they always simplify at least a little, since connected sum has smaller complexity than disjoint union.

In dimension $2$, Morse complexity of homology classes is proportional to their simplicial norm. Our goal in this paper is to study Morse complexity in higher dimensions, still motivated by the analogy with simplicial norm and questions of Gromov outlined in \cite[\S8$\frac{1}{2}$]{GrPC} and \cite[\S3.2]{GrSing}.  We will show that this notion is trivial in odd dimensions (as Gromov had speculated), and is nontrivial in all even dimensions (as Gromov knew).

Our main theorems are strongly motivated by some of the results of Gromov in \cite{GrVol}, but the proofs are very different: the upper bounds rely on surgery theory and the work of Quinn on open book decompositions \cite{Quinn} and the lower bounds rely on index theory. They also contrast strongly with the nontriviality of simplicial norm is all dimensions $>1$.
\begin{thmA} \label{intro:vanish} \ 
  \begin{enumerate}[(i)]
  \item The Morse complexity of homology classes vanishes in all odd dimensions.
  \item The Morse complexity of $\alpha \in H_d(M)$ depends only on its image
    in $H_d(K(\pi_1(M), 1))$.
  \end{enumerate}
\end{thmA}
This follows from a more fine-grained result about the Morse complexity of bordism classes, see \S\ref{S:bordism}.

Our next theorem is about locally symmetric manifolds, and indeed these are the only aspherical spaces above dimension $2$ for which we have nontrivial information. To state it, we need some terminology. If $\Gamma$ is a lattice in a semisimple group $G$, then $H^*(K \backslash G/\Gamma)$ has some elements which are represented by forms whose lift to the universal cover is $G$-invariant.  These classes form a subalgebra isomorphic to the cohomology of
the compact dual of $K\backslash G$, denoted $K\backslash G_u$; see Matsushima \cite{Mat}, and e.g.\ \cite[\S4.2]{Serre} for a useful exposition. We call these classes Matsushima classes.
\begin{thmA} \label{intro:main}
  Any homology class which pairs nontrivially with a Matsushima class
  (i.e.\ maps to a nontrivial class in $H_*(K\backslash G_u; \RR)$)
  has positive Morse complexity.
\end{thmA}
As a consequence of the theorem and its proof, we obtain the following topological conclusion:
\begin{cor}
  If $\pi_1M=\Gamma$ is a lattice in a semisimple group $G$, with
  $\chi(K\backslash G/\Gamma) \neq 0$, and a higher signature of $M$ associated
  to a Matsushima class is nonzero, then $M$ does not have an open book
  decomposition. In particular, such $K\backslash G/\Gamma$
  (e.g.\ even-dimensional hyperbolic manifolds) do not admit open book
  decompositions.
\end{cor}
According to a theorem of Lawson \cite{Lawson} (proved earlier by Alexander \cite{Alex2} in dimension 3), all closed odd-dimensional manifolds have open book decompositions, so some condition on the locally symmetric space is necessary.

We will show, as part of the proof of Theorem \ref{intro:vanish}, that in fact Morse complexity of a bordism class depends only on higher signatures.  Even then, not every higher signature gives a lower bound on Morse complexity, as one can see by considering tori.  Nevertheless, in all even dimensions we obtain nontrivial lower bounds from higher signatures associated to Chern characters of representations into isometry groups of quadratic or symplectic forms.  That Matsushima classes have this property is shown in Lusztig's thesis \cite[\S5]{Lusztig} for $\Sp(2n,\ZZ)$; Lusztig uses this to prove the Novikov conjecture for these classes, among other cases.  We extend this, unfortunately by considering all possible cases.\footnote{Gromov also had the idea of using Lusztig's thesis, and knew the nonvanishing of Morse complexity for the fundamental classes of the same locally symmetric spaces to which Theorem \ref{intro:main} applies. It seems that for these classes Gromov had some method not based on taxonomy in mind.}

However, when $\chi(K\backslash G/ \Gamma) = 0$, none of the cohomology of the symmetric space comes from flat bundles (see \cite[Appendix]{Milnor}) and the method completely fails (as it also must in odd dimensions).

We also note that when Morse complexity is positive, one can ask which of the Betti numbers of the representing manifolds must be positive. The answer is that it is exactly the middle-dimensional one---indicating that signature is inherently related to this problem. Moreover, if one represents homology classes using slightly singular manifolds---say, by stratified spaces whose singular set is $2$-dimensional, and defining Morse complexity by the sum of
the number of cells in the singular set and the Morse complexity of the complement---then, while the theory is unchanged in dimension 2, it vanishes above dimension 4 (we do not know what happens in the four dimensional case).

This paper was originally conceived of as an appendix to the paper \cite{MW}, which studies the simplicial complexity of nullcobordisms of PL manifolds (with bounded local complexity). The simplest case of the idea of this paper is actually the following result, showing that there is no relationship between the Morse complexity of a manifold and its nullcobordisms:
\begin{thmA} \label{intro:ex}
  There an infinite sequence of manifolds homotopy equivalent to $\RR P^7$ and
  cobordant to it over $B\ZZ/2\ZZ$, all of which have the same Morse complexity,
  such that the Morse complexity of any corresponding sequence of cobordisms is
  unbounded.
\end{thmA}
The invariant is simply the signature of the two-fold cover of the cobordism minus the signature of the cobordism. This can be thought of as a signature of local system, as used in the proof of Theorem \ref{intro:main}; however, since it is associated to a finite group, general principles imply that it will vanish for closed manifolds---indeed, that is why it is independent of the
cobordism between a given pair of manifolds.  So Theorem \ref{intro:main} will rest on more sophisticated flat bundles.

We close this introduction by pointing out that we feel that this paper leaves much more mysteries than it resolves. As Gromov points out, we do not know whether the fundamental classes of even-dimensional negatively curved manifolds always have positive Morse complexity. We know nothing about the Morse complexity of fundamental classes of irreducible even-dimensional locally symmetric manifolds not covered by Theorem \ref{intro:main}.  We also don't know whether homology classes of locally symmetric manifolds that do not survive in the homology of their compact duals have nontrivial Morse complexity, at least above dimension $2$.

\subsection*{Acknowledgements}
We thank the referees of \cite{MW} for making comments that ultimately contributed to improving the exposition of this paper.  The third author also thanks Gromov for useful email correspondence about this problem. The authors are supported by an NSERC Discovery Grant and NSF grants DMS-2236705 and DMS-2505738.

\section{Main results}
In this section we define a number of variations on Morse complexity and prove our main results.  We start with some remarks on (non-stable) Morse complexity of manifolds.

\subsection{Morse complexity of manifolds}
For a simply connected manifold $M$ of dimension $>4$, Smale
\cite[Theorem 6.1]{Smale} showed that the homology of $M$ determines $\MC(M)$ in
the most obvious way: the number of $i$-dimensional critical points is
\[p(i)+q(i)+q(i-1),\]
where $p(i)$ is the rank of the $i$th homology and $q(i)$ is the minimal number
of generators for the torsion.  Whitehead torsion shows that this cannot be
exactly correct in the non--simply connected case, at least for relative Morse
complexity: a manifold with boundary has relative Morse complexity zero if and
only if it is a product $N \times [0,1]$; but the property of being a product is
obstructed by the torsion and therefore isn't purely homotopy theoretic.  On the
other hand, we have the following result (which might be known):
\begin{thm}
  If $m=dim(M) > 5$, and $f:M' \to M$ is a simple homotopy equivalence, then
  $\MC(M') = \MC(M)$. \label{thm:she}
\end{thm}
\begin{proof}
  As in the proof of the $h$-cobordism theorem, we can assume that $M$ is given
  a self-indexing Morse function (i.e.\ a Morse function so that critical points
  of index $k$ arise before ones of index $k+1$).  We also assume that there is
  just one critical point each of indices $0$ and $m$.

  We construct a Morse function on $M'$ as follows.  In index $\leq 2$ we mimic
  the $2$-complex $M_2$ associated to the Morse function of $M$.  That is, we
  put $f(M_2)$ in general position, turning it into an embedding, and then take
  a suitable function on its regular neighborhood.  We do the same thing with
  the dual Morse function $-f$, creating a Morse function on a disjoint
  submanifold.

  Now what remains is a simple homotopy equivalence (which we still call $f$)
  between two manifolds with boundary, say $(A', \partial_\pm A')$ and
  $(A, \partial_\pm A)$, where the manifolds and their boundaries all have
  fundamental group $\pi$ and $A$ is given a Morse function (locally constant on
  the boundary) whose critical points all have index between $3$ and $m-3$.
  This Morse function yields a relative handlebody structure on the manifold.

  Now we perform an induction on the handles: at each step, we find a handle in
  $A'$ whose image is the lowest handle of $A$, and generate a new simple
  homotopy equivalence by cutting them both out.  Write
  $(H, \partial H) \subset A$ for this first handle, and
  $H^c=\overline{A \setminus H}$.  Now we apply \cite[Theorem 12.1]{Wall}.  By
  this theorem, $f$ can be homotoped so that it is transverse to $\partial H$,
  and so that it induces a simple homotopy equivalence of triads
  \[(A’; f^{-1}H, f^{-1}H^c) \to (A; H, H^c).\]
  Then $f^{-1}H$ is a manifold homotopy equivalent to
  \[(B^m; B^k \times S^{m-k-1} \cup S^{k-1} \times B^{m-k};  S^{k-1} \times S^{m-k-1});\]
  but by simply connected surgery on triads, any such homotopy equivalence is in
  fact a diffeomorphism, so $f^{-1}H \subset A'$ is a handle.

  Once one has pulled back all the handles of $A$ to $A'$, what lies between
  $\partial_+A'$ and this union is an $s$-cobordism, and therefore a product, so
  we can easily extend the Morse function to all of $A'$.
\end{proof}
\begin{rmk}
  If one does not first make the Morse function on $M$ self-indexing, it is
  possible to give examples where one cannot find a ``corresponding'' function
  on $M'$.
\end{rmk}

In the non--simply connected case, actually computing or estimating Morse complexity is not so easy.  If $N \to M$ is a $d$-fold cover, then obviously $\MC(N) \leq d\MC(M)$.  One can also use Betti numbers of finite covers to give lower bounds on $\MC(M)$.  These are typically hard to get information about except in characteristic $0$, where one usually gets this information from $L^2$ Morse inequalities, which describe the asymptotic behavior of finite covers.

One can ask when this obvious upper bound fails to be sharp, for instance:

\begin{qn}
  If $M$ is a locally symmetric space associated to a group $G$ without discrete
  series (equivalently $M=K \backslash G/\Gamma$ has Euler characteristic
  zero), is it true that $M$ has a sequence of finite covers $M_k$ of index
  $k$ such that $\MC(M_k)/k \to 0$?
\end{qn}

The optimistic conjecture that the answer is yes, even for general aspherical manifolds, whenever $L^2$ Betti numbers don't prevent this is disproved by the calculations in \cite{AvOS}.  On the other hand, a recent paper of Fr\k{a}czyk, Mellick and Wilkens \cite{FMW} shows that when $G$ has rank $>1$, the number of generators of the fundamental group grows sublinearly as a function of the covolume.  At the same time, Avramidi and Delzant \cite{AvDe} show that the Morse complexity of a hyperbolic manifold is bounded below by an increasing function of its injectivity radius.

If one asked instead about the number of simplices in a triangulation, of course
the theory of simplicial norm \cite{GrVol,LS} prevents the question from having
a positive answer.  In fact this theory shows that whenever there is a map
$M_k \to M$ of degree $k$, not necessarily a covering map, $M_k$ must have
$\Omega(k)$ simplices.

Similarly, one can ask:

\begin{qn}[{Gromov, cf.\ \cite[\S8$\frac{1}{2}$]{GrPC}}]
  When does a nice manifold $M$, e.g.\ negatively curved or locally symmetric,
  have a sequence $M_k \to M$ of degree $k$ manifolds mapping to it, with
  $\MC(M_k)/k \to 0$?
\end{qn}

The following proposition shows that the answer is \emph{always} for
odd-dimensional manifolds.  On the other hand, Gromov tried to explain to the third author why the opposite holds for the class of locally symmetric spaces with nonzero Euler characteristic; this result is included in Theorem \ref{intro:main}.

\begin{prop} \label{prop:openbook}
  If $\dim(M)$ is odd, then $M$ always has a sequence of $k$-fold finite
  branched covers (for any index $k$) with uniformly bounded Morse complexity.
\end{prop}
\begin{proof}
  This is an immediate consequence of Lawson's result \cite{Lawson} that all
  odd-dimensional closed manifolds have open book decompositions (in dimension
  $3$, this actually goes back to Alexander \cite{Alex2}).  The cyclic branched
  covers whose branch set is the binding of the open book obviously have the
  desired property.
\end{proof}
\begin{rmk}
  Even in even dimensions, using branched covers of ``almost open books'', one
  can get maps of degree $k$ from manifolds whose Betti numbers are bounded
  except in the middle dimension.
\end{rmk}

\begin{qn}
  Which $K \backslash G/\Gamma$ have open book decompositions?
\end{qn}

Obviously negative answers to Gromov's question give negative answers to this
question as well.  Below, we will use this to see geometrically that products of surfaces do not have open book decompositions, and the same is true for more complicated representation-theoretic reasons for all $K \backslash G/\Gamma$ where $G$ has discrete series.  (In principle, this should be a calculation in an algebraic theory invented by Quinn \cite{Quinn}; unfortunately the groups of non-symmetric quadratic forms involved seem very difficult to understand.  Ranicki \cite[\S29B]{Ranicki} has shown that Quinn's obstruction can be viewed as an element of a homology surgery group, but these groups are much less well-understood than $L$-groups of group rings, for which the Farrell--Jones conjecture provides a beautiful conjectural picture.)

\subsection{Morse complexity of bordism classes} \label{S:bordism}
Gromov's question above is phrased in terms of Morse complexity of covers, but it can be interpreted as asking about MC-minimal representatives for the fundamental class in real homology, as outlined in the introduction.  One can then ask analogous questions about integral homology or even bordism.  We consider both, starting with bordism.  Let $X$ be a finite complex.\footnote{This is an assumption that it would be desirable to weaken.}  We consider the bordism group $\Omega_d(X)$ of smooth oriented $d$-dimensional manifolds mapping to $X$ and try to estimate the size of its elements.  In light of Proposition \ref{prop:openbook}, the following strengthening of Theorem \ref{intro:vanish}(i) is not surprising:

\begin{prop} \label{prop:odd-d}
  For any odd $d$, all of $\Omega_d(X)$ has representatives of uniformly bounded
  Morse complexity.
\end{prop}
This follows from Lawson's theorem and the finite generation of $\Omega_d(X)$ together with the following lemma:
\begin{lem} \label{lem:openbook}
  Suppose that $M$ is a manifold with an open book decomposition and
  $f:M \to X$ represents an element $\alpha \in \Omega_d(X)$.  Then multiples
  of $\alpha$ have uniformly bounded Morse complexity.
\end{lem}
\begin{proof}
  Suppose first that $M$ has an open book decomposition with empty binding, in
  other words it fibers over $S^1$ and we can take a $k$-fold cyclic cover
  $p:\tilde M \to M$.  Let $h:\Sigma \to S^1$ be a cobordism over $S^1$ between
  $k$ copies of the identity map and a connected $k$-fold cover.  Then the
  pullback $N$ of the bundle $M \to S^1$ along $h$, equipped with the map
  $f \circ \tilde h:N \to X$, is a bordism between $kf$ and $f \circ p$.  The
  cyclic cover $\tilde M$ has the same Morse complexity as $M$.

  In general, the difference of two open books with the same binding is
  cobordant to a manifold that fibers over the circle with fiber the union of
  two pages, one from each book.  Applying this to $f$ and $-f$ gives a
  representative of $2[f]$ which fibers over the circle.

  This shows that even multiples of $\alpha$ have bounded Morse complexity.
  Odd multiples can be represented as $2n\alpha+\alpha$.
\end{proof}

For even $d$, the answer depends only on the fundamental group of $X$ by the
theorem below.  In particular, this proves Theorem \ref{intro:vanish}(ii).  This does not seem to be obvious even when $X$ is simply connected, but Ranicki's algebraic surgery theory implies:

\begin{thm} \label{thm:sig}
  For a finite complex $X$ and $d \neq 4$, the Morse complexity of a class in
  $\Omega_d(X)$ is determined up to a uniformly bounded error by the image of
  its higher signature class in $KO_d(K(\pi_1(X),1); \mathbb Q)$.
\end{thm}
The \strong{higher signature class} $s \in \bigoplus_i H_{d-4i}(K(\pi_1(X),1); \mathbb Q)$ of a map $f:M \to X$ is defined by
\[s=f_*(L(M) \cap [M]),\]
where $L(M)$ is the Hirzebruch $L$-class.  Rationally, this can be viewed as living in $KO_d$ since elements of $KO^d$ are determined by their Pontryagin classes.  (It is well-known that this is the homology $L$-class of the symbol of the signature operator on $M$.)

For example, if $X$ is simply connected, the result is that
$\MC(\beta) = \Sign(\beta) + O(1)$, where $\Sign$ is the signature of any
representative manifold for the class $\beta$.  That no other classes obstruct
is a consequence of Winkelnkemper's theorem \cite{Wink} (reproved by Quinn
\cite{Quinn}) that any simply connected manifold of dimension $>4$ with
signature $0$ is an open book. The signature obviously obstructs because
signature is a cobordism invariant and gives lower bounds for Betti numbers.

\begin{proof}[Proof of Theorem \ref{thm:sig}]
  Let $f:M^d \to X$ be a map.  We can ensure via surgeries that
  $\pi_1(M)=\pi_1(X)$.  We will show that every class in the kernel of the
  higher signature can be represented by manifolds of uniformly bounded Morse
  complexity.  The theorem follows.

  The class $L(M) \cap [M]$, when viewed in the appropriate $L$-group, is the
  \emph{symmetric signature} first constructed by Mischenko \cite{Mis}.  By
  results of Ranicki \cite[\S29]{Ranicki}, $f_*(L(M) \cap [M])$ has a further
  image, the \emph{asymmetric signature}, which is exactly the obstruction to
  the bordism class of $f:M \to X$ having an open book representative.
  Therefore, the kernel of the map
  \[[f:M \to X] \mapsto f_*([M] \cap L(M))\]
  is finitely generated by open books.  The result follows by Lemma
  \ref{lem:openbook}.
\end{proof}

In dimension $4$, we can use Kreck surgery to prove an analogous but slightly weaker result, which is still sufficient to prove Theorem \ref{intro:vanish}(ii) in dimension $4$.
\begin{thm}
  For a finite complex $X$, let $\alpha \in \Omega_4(X)$ be a class whose
  higher signature class is zero.  Then the Morse complexity of $n\alpha$ grows
  at most as $\log^2(n)$.
\end{thm}
\begin{proof}
  Let $f:M^4 \to X$ be a map representing $\alpha$.  As before, we can assume
  that $\pi_1(M)=\pi_1(X)=\pi$.

  Note that $\Omega_4(X)$ is rationally $H_4(X) \oplus H_0(X)$, with the latter
  summand representing the signature.  Thus, equivalently, $\Sign(M)=0$ and
  \[f_*[M] \in \ker(H_4(X;\QQ) \to H_4(B\pi;\QQ)).\]

  Suppose first that $M$ is spin.  Kreck \cite[Theorem C]{Kreck} (as elaborated
  in \cite[\S2.1.1]{KPT}) then shows that any $M'$ spin-cobordant to $M$ over
  $B\pi$ is diffeomorphic to $M$ after stabilization by connected sum with
  copies of $S^2 \times S^2$.

  Let $M'$ be obtained by surgery from $M \sqcup M$ which identifies the two
  copies of the fundamental group.  In particular, there is a natural degree
  $2$ map $g:M' \to M$, and $f \circ g$ again sends
  $[M'] \mapsto 0 \in H_*(B\pi;\QQ)$.  Thus $M'$ is stably bordant to $M$ over
  $B\pi$, giving a degree $2$ map $M \mathbin{\#} k(S^2 \times S^2) \to M$ for
  some $k$.  Iterating this procedure gives maps
  \[M \mathbin{\#} pk(S^2 \times S^2) \to M\]
  of degree $2^p$, so that $2^p[M] \in \Omega_4(X)$ has Morse complexity
  $2pk+MC(M)$.  Using the binary expansion of $n$, we get more generally that
  $n\alpha$ has Morse complexity $O(\log^2(n))$.

  Since classes with spin representatives form a subgroup of finite index, this
  completes the proof.
\end{proof}

\begin{rmk}
  Not all higher signatures give rise to lower bounds on Morse complexity.  For
  example, Gromov's question can be rephrased as asking whether the top
  cohomology class does, in the case that $X$ is a closed oriented manifold.
  When $X=T^n$, it obviously does not, for example because tori have
  self-covers of arbitrary degree.
\end{rmk}

\subsection{Lower bounds via local systems}\label{sec:local-systems}
Now we work our way towards the proof of Theorem \ref{intro:main}.  A source of higher signatures that \emph{do} give lower bounds is from local systems.  If $\mathcal L$ is a local system of (perhaps skew-)symmetric or Hermitian quadratic forms on $K(\Gamma,1)$, then for a $2k$-manifold $M$ equipped with a map $M \to K(\Gamma,1)$ one gets an inner product pairing on $H^k(M; \mathcal L)$ whose signature is an invariant of elements of $\Omega_{2k}(K(\Gamma,1))$.  This signature gives a lower bound on the rank of the $k$th homology and therefore on $\MC(M)$.  

The signature of local systems arises in showing that signature is not
multiplicative on bundles \cite{Atiyah,Meyer}.  Using Atiyah's examples of this phenomenon, we show the following proposition; its proof generalizes to the much more powerful Theorem \ref{thm:MC-lower-bound}.
\begin{prop} \label{prop:PiSigmas}
  Let $\Sigma$ be a surface of genus at least $2$.  Then the signature
  associated to the fundamental cohomology class $[\Sigma]$ gives a lower bound
  on Morse complexity of bordism classes.
\end{prop}

\begin{proof}
  Let $f:M^{4n+2} \to \Sigma$ be a map, and let $S$ be another surface of genus
  at least $2$.  According to Atiyah \cite[\S4]{Atiyah}, bundles $E \to B$ with
  fiber $S$ satisfy the equation
  \[\Sign(E)=\langle c_1(\xi) \cup L(B), [B]\rangle,\]
  where $\xi$ is the induced vector bundle over $B$ with fiber
  $H^1(S; \mathbb R)$.

  Let $p:E \to \Sigma$ be an $S$-bundle such that $E$ has nonzero signature, as
  constructed by Atiyah.  Then the above identity implies, by functoriality,
  that
  \[\langle f^*[\Sigma] \cup L(M), [M] \rangle = \frac{\Sign(f^*E)}{\Sign(E)}.\]
  By \cite[Satz~I.2.2]{Meyer} $\Sign(f^*E)$ gives a lower bound on
  $\rk(H^{2n+1}(M))\rk(H^1(S))$, and therefore a lower bound on the Morse
  complexity.
\end{proof}

On the other hand, suppose $X$ is a product of two surfaces, and let $\omega$ be a cup product of one-dimensional classes from different surfaces.  Then the associated signature \emph{does not} give a lower bound on Morse complexity: $\omega$ pairs nontrivially with a torus, and the signature of maps to this torus does not give a lower bound on Morse complexity, as already discussed.

Lusztig, in his thesis \cite[\S5]{Lusztig}, showed that Chern characters of local systems on $\Sp(2n, \mathbb Z)$ surject onto the rational cohomology of $B\Sp(2n, \mathbb R)$.  By the same argument as in Prop.~\ref{prop:PiSigmas}, signatures associated to such cohomology classes give lower bounds on Morse complexity.  In particular, if a manifold maps to a nontrivial cycle in this group, then its fundamental class in real homology has nonzero complexity.  Gromov, in \cite[\S8$\frac{1}{2}$,~pp.~139--140]{GrPC}, gives other examples, including all products of even-dimensional hyperbolic manifolds.

As Gromov suggests, Lusztig's calculations actually generalize to all $G$ which have discrete series representations (equivalently, $K \backslash G/\Gamma$ of nonzero Euler characteristic).  The calculations are given in Section \ref{S:calc}.

To formulate the result, it is necessary to discuss a well-known aspect of the homology of locally symmetric spaces (see e.g.~the discussion of the Matsushima formula in \cite[\S4.2]{Serre}; our target is the piece corresponding to the trivial representation).  Every symmetric space $K \backslash G$ of noncompact type has a compact dual, which we denote $K \backslash G_u$---and the real cohomology of a compact locally symmetric space $K \backslash G/\Gamma$ always contains $H^*(X)$ as a subalgebra.  This subalgebra consists of the cohomology classes whose harmonic representatives, when lifted to the universal cover $K \backslash G$, are invariant under the action of the whole group $G$.  We are interested in the map on homology dual to this inclusion.

In these terms, the Lusztig-type calculations described above imply Theorem \ref{intro:main}, stated with more precision as follows:
\begin{thm} \label{thm:MC-lower-bound}
  Let $M$ be an $m$-manifold such that:
  \begin{itemize}
  \item $\pi_1(M)$ is (or maps to) a lattice $\Gamma$ in a Lie group $G$
    admitting discrete series representations.
  \item The higher signature class of $M$ has nontrivial image in the homology
    of the dual symmetric space of $K \backslash G$, i.e.~in
    $\bigoplus_i H_{m-4i}(K \backslash G_u;\mathbb Q)$.
  \end{itemize}
  Then the Morse complexity of the fundamental class $[M]$ can be bounded below
  by this image, and is not stably trivial.  In particular, $M$ does not have
  an open book decomposition.

  More generally, such a lower bound holds whenever $M \to BG$ is a map from an
  even-dimensional manifold $M$ and the higher signature class has nontrivial
  image.  Consequently, such an $M$ likewise does not admit an open book
  decomposition.
\end{thm}

In particular, even-dimensional hyperbolic manifolds satisfy this theorem.  More generally, one can consider locally symmetric manifolds associated to lattices in $O(m,n)$, if $mn$ is even, $SU(m,n)$, $Sp(m,n)$, $SO^*(2n)$, and $Sp_{2n}$, as well as certain exceptional Lie groups.

We now return to Gromov's question of lower-bounding Morse complexity in a homology class, rather than a bordism class.  Rationally, $H_n(X)$ is isomorphic to the group of \emph{framed} bordism classes of $n$-manifolds in $X$, since the latter is $\pi_n^s(X)$.  In other words, every homology class can be stably represented by a manifold with trivial $L$-class, and therefore by Theorem \ref{thm:sig} the stable Morse complexity of a class in $H_n(X)$ is determined by its image in $H_n(K(\pi_1(X),1);\QQ)$.

In particular, for example, the stable Morse complexity is always trivial in simply connected manifolds, and manifolds with nonzero signature may have fundamental classes with trivial stable Morse complexity.  This can be seen easily for, e.g., $\CC P^2$, which actually has nontrivial self-maps of degree $d^2$ for any $d$.

On the other hand, the calculation used to prove Theorem \ref{thm:MC-lower-bound} also gives lower bounds, given a manifold $M$ satisfying the hypotheses, on the stable Morse complexity of homology classes that pair nontrivially with the cohomology of the compact dual of $K \backslash G$.

\subsection{Morse complexity of nullcobordisms}
Gromov \cite[\S5$\frac{5}{7}$.I]{GrPC} suggested studying the minimal Morse complexity of a nullcobordism of a manifold $M$ and sketched a way of showing that, when the tangent bundle of $M$ has trivial characteristic classes, then this is linear in the Morse complexity of $M$.  Here is a result in this direction that derives from a simple case of this sketch:
\begin{thm}
  If $M$ is a framed $n$-manifold for $n \not\equiv 3 \mod 4$, then it is
  nullcobordant via
  \[\leq 2\rk \pi_1(M)+\sum_{2 \leq i<n/2}(\rk H_i(M)+2\#\{\text{torsion generators of }H_i(M)\})+\text{const}_n\]
  surgeries.
\end{thm}
\begin{proof}
  One follows the arguments of Kervaire and Milnor \cite{KM} to surger the
  manifold to a homotopy sphere.  In dimension $1$, one kills a generating set
  for $\pi_1$, but this may produce a corresponding set of extra generators for
  $H_2$.  In higher dimensions $k$, one kills a generating set for $H_k$, but
  killing torsion classes produces a $\ZZ$-summand one dimension higher.

  The nullcobordism of the resulting exotic sphere requires at most
  $\text{const}_n$ handles since there are a finite number of possibilities.
\end{proof}
When $n \equiv 3 \mod 4$, Kervaire and Milnor's construction \cite[\S6]{KM} in the middle dimension does not necessarily kill torsion one generator at a time, but merely reduces its size.  So while one can still bound the number of handles required, there is an extra term based on the number of prime factors in the cardinality of the torsion subgroup of $H_{\frac{n-1}{2}}(M)$, which is not bounded by the number of handles of $M$.  It remains to be seen how far the linearity result can be generalized.

In contrast, we now give the proof of Theorem \ref{intro:ex}, showing that in general, the minimal Morse complexity of at least a nullbordism over $B\pi_1M$ need not even be bounded as a function of the Morse complexity of $M$.

\begin{ex}
  According to Browder and Livesay \cite{BL} there are infinitely many smooth
  manifolds simple homotopy equivalent to $\mathbb{RP}^{4k+3}$ that are cobordant
  over $B\mathbb Z_2$.  (Since $\Omega_i(B\mathbb Z_2)$ is finite for odd $i$,
  the last condition does not need to be added explicitly, but nevertheless such
  cobordisms can be explicitly constructed.)  Chang and Weinberger \cite{ChW}
  extend this construction to any oriented closed manifold $M^{4k+3}$ whose
  fundamental group contains torsion.

  By Theorem \ref{thm:she}, manifolds in this family all have the same Morse
  complexity.  However, by work of Hirzebruch \cite{Hirz}, the Browder--Livesay
  invariant can be equivalently defined as twice the signature of the
  cobordism between them minus the signature of its twofold cover (in the more
  general situation of Chang and Weinberger, this is replaced by an $L^2$
  signature of the cobordism).  This gives a lower bound on the Morse complexity
  of the cobordism, and in particular shows that for non--simply connected
  targets one cannot always bound the Morse complexity of the smallest
  nullcobordism of a manifold in terms of the Morse complexity of the manifold.

  Note that this argument can be viewed as a twisted signature argument, like
  the proofs of Prop.~\ref{prop:PiSigmas} and Theorem \ref{thm:MC-lower-bound},
  using bundles with finite structure group; although in the setting of closed
  manifolds, using finite structure group would give trivial invariants since
  signature is multiplicative in covers.  In contrast, similar examples of
  $M^{4k+3}$ that are hard to cobord over $B\pi$ can be constructed in the
  setting of lattices in $BG$ using the bundles that arise in that setting, in
  particular examples where the group is torsion-free.
\end{ex}

One can give other examples with torsion-free fundamental group using surgery theory and products of surface groups.  Let $M$ be a $5$-manifold whose fundamental group $\pi$ is the product of three genus $2$ surface groups.  The product of these surfaces with the Milnor manifold constructed from $k$ copies of $E_8$, mapping to $S^8$, gives an element of the group $L_{14}(\pi) \cong L_6(\pi)$.  Using Wall's realization theorem for this element of $L_6(\pi)$, one obtains a cobordism from $M$ to $M'$, a simple homotopy equivalent manifold, whose surgery obstruction is the given element.  Now the argument is the same as above: the normal cobordism ``needs'' to have enough handles that a product of surface bundles over it can have large signature.

\subsection{CW complexity}
An additional complexity measure in a similar spirit is the CW complexity of a cell complex: the minimal number of cells required to construct it (up to homotopy equivalence, naturally).  Similar to the Morse complexity of bordism classes, one can also define the CW complexity of a class in $H_d(X)$ to be the minimal CW complexity of its pseudomanifold representatives.  This is clearly bounded above by the Morse complexity of the homology class (times a dimensional constant).  Under closer examination, however, it turns out to be uniformly bounded in most cases when $X$ is a finite complex.  In all odd dimensions, this follows from Prop.~\ref{prop:odd-d} and the fact that bordism rationally surjects onto homology.  In even dimensions, there is a bit more work to do:
\begin{thm}
  For any even $d\ge6$ and any finite CW complex $X$, elements of $H_d(X)$ have
  uniformly bounded CW complexity.
\end{thm}
\begin{proof}
  It suffices to show that for $\alpha \in H_d(X)$, integer multiples $k\alpha$
  have uniformly bounded CW complexity.  Then the theorem follows by finite
  generation.  Indeed we need only show this for $\alpha$ which have manifold
  representatives, since the subgroup of such $\alpha$ has finite index.

  Let $f:M \to X$ be a representative of $\alpha$, and fix a cell structure for
  $M$.  Let $(M^c, \partial M^c)$ be the complement of a tubular neighborhood
  of the $2$-skeleton $M^{(2)}$.  Then $\partial M^c$ has an open book
  decomposition since it is odd-dimensional.  Moreover, by Quinn's
  \cite[Theorem 1.1(4)]{Quinn}, since $\pi_1(\partial M^c)=\pi_1(M^c)$, one can
  find an open book decomposition on $\partial M^c$ which extends to $M^c$.

  The proof of Lemma \ref{lem:openbook} holds verbatim for bordism of pairs, in
  particular for the identity map $(M^c, \partial M^c) \to (M^c, \partial M^c)$.
  Thus for every $k$, we have a map
  $h_k:(M_k, \partial M_k) \to (M^c,\partial M^c)$ such that the $M_k$ have
  bounded Morse complexity.  Let $g:(M^c, \partial M^c) \to (M,M^{(2)})$ extend
  the projection map from the boundary of the tubular neighborhood to $M^{(2)}$.
  Then $P_k=M_k \cup_{gh_k} M^{(2)}$ is a pseudomanifold with singular set
  $M^{(2)}$, and $fgh_k$ factors through a representative $p_k:P_k \to X$ of
  $k\alpha$.
\end{proof}
CW complexity of $H_2(X)$ is equivalent to simplicial volume and hence nontrivial.  Dimension $4$ rests as the only remaining open case.

\section{Rational surjectivity of the Chern signature} \label{S:calc}
We now discuss Theorem \ref{thm:MC-lower-bound} and its proof in more detail.

Let $G$ be a semisimple, real Lie group and fix a maximal compact subgroup $K$. Write $R(G)$ for the Grothendieck ring generated by isomorphism classes of finite-dimensional, complex Hermitian representations of $G$. The Chern signature is the composite ring homomorphism
\[\gamma:\xymatrix@R=0pc{R(G) \ar[r]^-{\mathrm{sig}} & K(BG) \ar[r]^-{\ch} & H^{\ev}(BG) \\
  V \ar@{|->}[r] & \mathbb V^+-\mathbb V^- \ar@{|->}[r] & \ch(\mathbb V^+-\mathbb V^-).}
\]
To define $\mathrm{sig}$, we write $V=V^+\oplus V^-$ for the decomposition of $V$ into positive/negative-definite summands with respect to the action of $K<G$, and we write $\mathbb V^{\pm}:=(EG\times V^{\pm})/G \to BG$ for the associated vector bundles. As usual, $\ch$ denotes the Chern character. We write $H^{\ev}(X):=\prod_{i\ge0} H^{2i}(X;\mathbb Q)$, and view this as a topological space with the product topology. The main result of this section is as follows. 

\begin{thm}\label{thm:rational-surjectivity}
Let $G_0$ be a semisimple, real Lie group. Assume that $G_0$ has discrete series representations. Then there is a group $G$ with finite center and the same Lie algebra as $G_0$ such that the morphism $\gamma\otimes\mathbb Q: R(G)\otimes\mathbb Q \to H^{\ev}(BG)$ has dense image. 
\end{thm}

Theorem \ref{thm:rational-surjectivity} generalizes Lusztig's result for $G=\Sp_{2n}(\RR)$, as discussed in \S\ref{sec:local-systems}.  Theorem \ref{thm:rational-surjectivity} implies Theorem \ref{thm:MC-lower-bound}. To explain this, we use the following Proposition \ref{prop:UmodK} and Corollary \ref{cor:dense-image}.

\begin{prop}\label{prop:UmodK}
Let $K\subset U$ be compact, connected simple real Lie groups. If $\mathrm{rank}(K)=\mathrm{rank}(U)$, then the odd Betti numbers of $U/K$ are zero, and the natural map $H^*(BK)\to H^*(U/K)$ is surjective. 
\end{prop}

\begin{proof}
The first statement is explained in \cite[\S26]{borel}. For the second statement, consider the Serre spectral sequence for the fibration $U/K\to BK\to BU$. Since $U/K$ and $BU$ have cohomology concentrated in even degrees, there are no nontrivial differentials out of the groups $E_r^{0,q}$ for $r\ge2$ and $q\ge0$. Thus $E_2^{0,q}\cong H^q(U/K)$ survives to the $E_\infty$ page, for each even $q\ge0$, which implies that $H^*(BK)\to H^*(U/K)$ is surjective. 
\end{proof}

For the next corollary, recall that for a symmetric space $G/K$, the compact dual symmetric space has the form $U/K$, where $U$ is a maximal compact subgroup of the complexification of $G$.

\begin{cor}\label{cor:dense-image}
Let $G$ be as in the statement of Theorem \ref{thm:rational-surjectivity}. 
Write $U/K$ for the compact dual symmetric space of $G/K$. 
Then the composition 
\[R(G)\otimes\QQ\to H^{\ev}(BG)\cong H^{\ev}(BK)\to H^{\ev}(U/K)\cong H^*(U/K)\]
is surjective. 
\end{cor}

\begin{proof}[Proof of Corollary \ref{cor:dense-image}]
By a theorem of Harish-Chandra (see e.g.\ \cite[Thm.\ 12.20]{knapp}), $G$ has discrete series representations if and only if rank$(G)=$ rank$(K)$, where rank denotes the maximum dimension of a Cartan subalgebra of the Lie algebra. Furthermore, the rank of $G$ is the rank of the maximal compact subgroup $U$ in the complexification of $G$. 
Now $\mathrm{rank}(K)=\mathrm{rank}(U)$ implies that $H^{\ev}(U/K)\cong H^*(U/K)$ by Proposition \ref{prop:UmodK} and because $U/K$ is a manifold. To show surjectivity, fix $x\in H^{2k}(U/K)$. By Proposition \ref{prop:UmodK}, we can lift $x$ to $x'\in H^{2k}(BG)$. Because we assume $\gamma$ has dense image, we can find $x''\in R(G)\otimes\QQ$ such that $\gamma(x'')$ agrees with $x'$ up to the dimension of $U/K$. Then $x''$ and $x'$ have the same image in $H^*(U/K)$, which is $x$ by construction. 
\end{proof}

To deduce Theorem \ref{thm:MC-lower-bound}, by the preceding results, if $\gamma\otimes\mathbb Q: R(G)\otimes\mathbb Q \to H^{\ev}(BG)$ has dense image, then each cohomology class in the image of the Matsushima homomorphism $H^*(U/K)\to H^*(K\backslash G/\Gamma)$ is the Chern signature of a local system on $K(\Gamma,1)\simeq K\backslash G/\Gamma$. Accordingly, the signatures associated to these cohomology classes give lower bounds on the Morse complexity of elements in $\Omega_*(K(\Gamma,1))$ whose higher signature class has nontrivial image in the homology of $U/K$. Compare with the discussion in \S\ref{sec:local-systems}. The same argument works if we replace $\Gamma$ by a lattice in a finite cover $G_0\to G$. If instead $\Gamma$ is a lattice in a (finite) quotient $G\twoheadrightarrow G_0$, then since $G$ is linear the pre-image of $\Gamma$ in $G$ is residually finite and projects to $\Gamma$ with finite kernel, so a finite-index subgroup of $\Gamma$ lifts to $G$, which suffices.



\subsection{Generalities related to the proof of Theorem \ref{thm:rational-surjectivity}}\label{sec:setup}
Our proof proceeds by direct computation for each simple $G$ with discrete series representations. (A less computational approach would be most desirable.) Applying Harish-Chandra's criterion $\mathrm{rank}(G)=\mathrm{rank}(K)$, of the classical simple $G$, those with discrete series are 
$\Sp_{2n}(\mathbb R)$, $\SO^*(2n)$, $\SU(p,q)$, $\Sp(p,q)$, and also $\SO(p,q)$ with $pq$ even. Among the exceptional $G$, we obtain the list in Table \ref{except-table}. 
\begin{table}
  \centering
\begin{tabular}{c|c|c|c|c}
$G$& $\text{Lie}(K)$& rank$(G)$ & rank$(K)$ & discrete series?\\\hline
$G_{2(2)}$&$\mathfrak{su}(2)\times\mathfrak{su}(2)$&2&2&yes\\
$F_{4(4)}$ & $\mathfrak{sp}(3)\times\mathfrak{su}(2)$&4&4&yes\\
$F_{4(-20)}$&$\mathfrak{so}(9)$&4&4&yes\\
$E_{6(6)}$&$\mathfrak{sp}(4)$&6&4&no\\
$E_{6(2)}$&$\mathfrak{su}(6)\times\mathfrak{su}(2)$&6&6&yes\\
$E_{6(-14)}$&$\mathfrak{so}(10)\times\mathfrak{so}(2)$&6&6&yes\\
$E_{6(-26)}$&$\mathfrak f_4$&6&4&no\\
$E_{7(7)}$&$\mathfrak{su}(8)$&7&7&yes\\
$E_{7(-5)}$&$\mathfrak{so}(12)\times\mathfrak{su}(2)$&7&7&yes\\
$E_{7(-25)}$&$\mathfrak e_6\times\mathfrak{so}(2)$&7&7&yes\\
$E_{8(8)}$&$\mathfrak{so}(16)$&8&8&yes\\
$E_{8(-24)}$&$\mathfrak e_7\times\mathfrak{su}(2)$&8&8&yes\\
\end{tabular}
\caption{Exceptional simple Lie groups.} \label{except-table}
\end{table}



While our proof requires a separate computation for each $G$, there are common ingredients. Some intermediate problems related to proving Theorem \ref{thm:rational-surjectivity} include computing $H^{\ev}(BG)$, finding ``sufficiently many" Hermitian representations of $G$, and determining criteria for density of a subring of $H^{\ev}(BG)$ that will be useful in our setting. In this section we address these, and provide the framework to understand the proof of Theorem \ref{thm:rational-surjectivity} in the next section.

\subsubsection*{ The rings $H^{\ev}(BG)$}\label{sec:coho}

In general, the ring 
\[H^*(BG)\cong H^*(BK)\equiv\bigoplus_{i\ge0}H^i(BK;\mathbb Q)\] is isomorphic to $H^*(BT)^W$, where $T\subset K$ is a maximal torus and $W=N_K(T)/T$ is the Weyl group. In each case, $H^*(BK)$ is isomorphic to a polynomial ring in finitely many variables, so $H^{\ev}(BK)\cong H^{\ev}(BG)$ is the ring of formal power series in the same variables. The computations we need are given in Table \ref{coh-table}. 
\begin{table}
\centering

  {\renewcommand{\arraystretch}{1.5}
\begin{tabular}{|c|c|c|c|c|}
\hline
$G$& $T$& $W$& $H^*(BT)^W$\\\hline\hline
$\U(n)$&$\U(1)^{\times n}$&$S_n$&$\mathbb Q[x_1,\ldots,x_n]^{S_n}$\\\hline
$\SU(n)$&$\U(1)^{\times (n-1)}$&$S_n$&$\frac{\mathbb Q[x_1,\ldots,x_n]^{S_n}}{(x_1+\cdots+x_n)}$\\\hline
$\Sp(n)$&$\U(1)^{\times n}$&$(\mathbb Z_2)^n\rtimes S_n$&$\mathbb Q[x_1^2,\ldots,x_n^2]^{S_n}$\\\hline
$\SO(2n+1)$&$\SO(2)^{\times n}$&$(\mathbb Z_2)^n\rtimes S_n$&$\mathbb Q[x_1^2,\ldots,x_n^2]^{S_n}$\\\hline
$\SO(2n)$&$\SO(2)^{\times n}$&$(\mathbb Z_2)^{n-1}\rtimes S_n$&$\langle\mathbb Q[x_1^2,\ldots,x_n^2]^{S_n},x_1\cdots x_n\rangle$\\\hline
$E_6$&&&$\mathbb Q[I_2,I_5,I_6,I_8,I_9,I_{12}]$\\\hline
$E_7$&&&$\mathbb Q[I_2,I_6,I_8,I_{10},I_{12},I_{14},I_{18}]$\\\hline
\end{tabular}}
\caption{Cohomology rings of $BG$ for some $G$; see e.g.\ \cite[Ch.\ 1]{adams}.} \label{coh-table}
\end{table}

For the classical groups, we can take $T$ to be a (block) diagonal subgroup. Then the group $W$ is a group of (signed) permutation matrices, and the variable $x_i\in H^2(BT)$ is defined as the first Chern class $x_i=c_1(BL_i)$ of the line bundle $BL_i:BT\to B\U(1)$ induced by the $i$-th coordinate function $L_i: T\to \U(1)\cong\SO(2)$. 

To further explain the exceptional group examples, first consider $E_7$. We choose $T\hookrightarrow E_7$ to be the standard torus that factors through the subgroup $\big(\Spin(12)\times\SU(2)\big)/\mathbb Z_2\hookrightarrow E_7$. Here $\big(\Spin(12)\times\SU(2)\big)/\mathbb Z_2$ is the maximal compact subgroup of the real form $E_{7(-5)}$ and the complexification $E_{7(-5)}\to E_7(\mathbb C)$ restricts to an embedding between maximal compact subgroups $\big(\Spin(12)\times\SU(2)\big)/\mathbb Z_2\hookrightarrow E_7$. There are various choices of generating polynomials for $H^*(BT)^W$. One choice is the polynomials $I_2,I_6,I_8,I_{10},I_{12},I_{14},I_{18}$, where 
\[I_{2k}=\sum_{i=1}^62\big[(x_i+y)^{2k}+(x_i-y)^{2k}\big]+\sum_{\substack{\epsilon\in\{\pm1\}^6\\ \prod\epsilon_i=-1}}\left[\frac{1}{2}(\epsilon_1x_1+\cdots+\epsilon_5x_5)\right]^{2k}.\]
Up to a scalar, $I_{2k}$ is equal to the degree-$2k$ term in the power series $\sum e^{c_1(B\lambda_i)}$, where $\lambda_i:T\to\U(1)$ are the weights of the fundamental representation of $E_7$. The fact that the polynomials $I_2,I_6,I_8,I_{10},I_{12},I_{14},I_{18}$ freely generate $H^*(BT)^W\cong H^*(BE_7)$ follows from \cite{mehta}, as explained in \cite[\S3.3]{pontlocsym}. 

A similar discussion applies to $E_6$, where we choose $T\hookrightarrow E_6$ to be the standard torus that factors through the subgroup $\big(\Spin(10)\times\SO(2)\big)/\mathbb Z_4\hookrightarrow E_6$, where $\big(\Spin(10)\times\SO(2)\big)/\mathbb Z_4$ is the maximal compact subgroup of the real form $E_{6(-14)}$. 
In this case $H^*(BT)^W$ is the polynomial algebra freely generated by $I_2,I_5,I_6,I_8,I_9,I_{12}$, where 
\[I_k=(-4y)^k+\sum_{i=1}^5\big[(2y-x_i)^k+(2y+x_i)^k\big]+\sum_{\substack{\epsilon\in\{\pm1\}^5\\ \prod\epsilon_i=1}}\left[\frac{1}{2}(\epsilon_1x_1+\cdots+\epsilon_5x_5)-y\right]^k.\]

\subsubsection*{Hermitian representations, weights, and Chern character}\label{sec:reps}

The proof of Theorem \ref{thm:rational-surjectivity} requires finding Hermitian representations $V$ of $G$ and computing $\gamma(V)=\ch(\mathbb V^+-\mathbb V^-)$. 

In terms of finding Hermitian representations, for the classical simple groups $G$, the standard representation of $G$ is obviously Hermitian in every case of interest. We will see that this representation, its exterior powers, and its orbit under the Adams operations, are \emph{almost} enough to prove the theorem (the exception is $\SO(p,q)$ where we will also need to use spin representations). For the exceptional groups $G$, we will either use the (minimal positive-dimensional) fundamental representation or the adjoint representation, together with exterior powers and Adams operations. The adjoint representation is always Hermitian via the Killing form. To see that the fundamental representation $U$ is Hermitian, recall that it is equivalent to show that the dual and conjugate representations are isomorphic $U^*\cong \overline{U}$. This isomorphism is automatic for exceptional $G$ that are not real forms of $E_6$ for the simple reason that $U$ is determined uniquely by its dimension. For real forms of $E_6$ as small additional argument is needed. 

We recall, for later use, the general relationship between weights of a representation and its Chern character. Specific computations will appear in the proof of Theorem \ref{thm:rational-surjectivity}. Fix a $G$-representation $V\cong\mathbb C^m$. When restricted to a maximal torus $T\subset K\subset G$, the representation is diagonalizable and has weights $\lambda_i:T\to\mathbb C^\times$ for $1\le i\le m$. These weights determine line bundles $B\lambda_i:BT\to B\mathbb C^\times=B\GL_1(\mathbb C)$, and the Chern character is 
\[\ch(V)=\sum_{i=1}^m e^{c_1(B\lambda_i)}\in H^{\ev}(BT)^W\cong H^{\ev}(BG),\] where $c_1(B\lambda_i)\in H^2(BT)$ is the first Chern class. Compare with  \cite[\S9]{BH}. 

\subsubsection*{ Dense image for $R(G)\otimes\mathbb Q\to H^{\ev}(BG)$}

Here we focus on ways to show that a subalgebra of $H^{\ev}(BG)$ is dense. For this we will make good use of Adams operations. 

We work in a formal power series ring $H=\mathbb Q[[t_1,\ldots,t_r]]$, where each $t_i$ has a formal degree $d_i$. The degree of a monomial $t_I^{\alpha}=t_{i_1}^{\alpha_{i_1}}\cdots t_{i_m}^{\alpha_{i_m}}$ is defined as $\alpha_{i_1}d_{i_1}+\cdots+\alpha_{i_m}d_{i_m}$. We assume that each $d_i$ is even (as is true for $H=H^{\ev}(BG)$).

Now we make some remarks about some additional structure for the rings $R(G)$, $H^{\ev}(BG)$ and the map $\gamma: R(G)\otimes\mathbb Q\to H^{\ev}(BG)$. The ring $R(G)$ is a $\lambda$-ring (direct sum, tensor product, exterior powers). As a $\lambda$-ring, $R(G)$ has Adams operations $\psi^k:R(G)\to R(G)$, which are ring homomorphisms and are given by $\psi^k(L)=L^{\otimes k}$ when $L$ is 1-dimensional. Power series rings $\mathbb Z[[x_1,\ldots,x_n]]$ are also $\lambda$-rings, but this does not apply to $H^{\ev}(BG)$ (since we take cohomology with $\mathbb Q$ coefficients), and the map $\gamma$ is not a map of $\lambda$-rings. Even so, $H^{\ev}(X)$ does have Adams operations, which are simple to describe: for $p\in H^{\ev}(BG)$, $\psi^k(p)$ acts on the degree-$2d$ term of $p$ by multiplication by $k^d$. The map $\gamma\otimes\mathbb Q$ commutes with Adams operations in the following sense. 

\begin{lem}\label{lem:adams-commute}
Let $V$ be a Hermitian representation of $G$. Recall that $V=V^+\oplus V^-$ denotes the decomposition in positive/negative definite subspaces with respect to $K<G$, and $\mathbb V^{\pm}$ denote the induced bundles over $BK$. Then 
\[\gamma(\psi^kV)=\psi^k\big(\ch \mathbb V^+\big)+(-1)^k\psi^k\big(\ch \mathbb V^-\big).\]
\end{lem}

\begin{proof}
To aid in this computation we use the following diagram, where the horizontal maps are those used to define $\gamma$, and the vertical maps are given by restriction to a fixed maximal torus $T<K<G$. 
\[\begin{xy}
(-25,0)*+{R(G)}="A";
(0,0)*+{K(BG)}="B";
(25,0)*+{H^{\ev}(BG)}="C";
(-25,-15)*+{R(T)}="D";
(0,-15)*+{K(BT)}="E";
(25,-15)*+{H^{\ev}(BT)}="F";
{\ar"A";"B"}?*!/_3mm/{};
{\ar "B";"C"}?*!/_3mm/{};
{\ar "D";"E"}?*!/_3mm/{};
{\ar "E";"F"}?*!/_3mm/{};
{\ar "A";"D"}?*!/^3mm/{i^*};
{\ar "B";"E"}?*!/^3mm/{i^*};
{\ar@{^{(}->} "C";"F"}?*!/_3mm/{};
\end{xy}\]
The right vertical map is injective, so it suffices to check the equality in $H^{\ev}(BT)$. Thus, it suffices to show (after using the naturality of $\gamma$, $\psi^k$, and $\ch$) that 
\begin{equation}\label{eqn:adams-commute}\gamma\big (\psi^k (i^*V)\big)=\psi^k\big(\ch (i^*\mathbb V^+)\big)+(-1)^k\psi^k\big(\ch (i^*\mathbb V^-)\big).\end{equation}

We separately compute the left- and right-hand sides of (\ref{eqn:adams-commute}) and see that they agree. 

For the left-hand side of (\ref{eqn:adams-commute}), the $T$-representation $i^*V$ splits as a sum of 1-dimensional representations 
\[i^*V=L_1\oplus \cdots\oplus L_p\oplus J_1\oplus\cdots\oplus J_q,\]
where, in our notation, the Hermitian form is positive on the $L_i$ and negative on the $J_i$. By properties of the Adams operations (group homomorphism, action on 1-dimensional representations), we have  
\[\psi^k(i^*V)=L_1^{\otimes k}\oplus \cdots\oplus L_p^{\otimes k}\oplus J_1^{\otimes k}\oplus\cdots\oplus J_q^{\otimes k}.\]

The induced Hermitian form on $L_i^{\otimes k}$ is positive definite and the Hermitian form on $J_i^{\otimes k}$ is positive/negative-definite when $k$ is even/odd (recall that the Hermitian form induced on a tensor product is given by multiplication, i.e.\  $\beta_{X\otimes Y}(x\otimes y, x'\otimes y')=\beta_X(x,x')\cdot \beta_Y(y,y')$). Using this, if we write $\mathbb L_i,\mathbb J_j$ for the line bundles over $BT$ associated to $L_i,J_j$, then we have  
\[\gamma\big(i^*\psi^k(V)\big)=\sum_i e^{kc_1(\mathbb L_i)}+(-1)^k\sum_je^{kc_1(\mathbb J_j)}.\]

For the right-hand side of (\ref{eqn:adams-commute}), the bundles $\mathbb V^\pm$ over $BK$ pull back to $BT$ and decompose $i^*\mathbb V^+=\mathbb L_1\oplus\cdots\oplus \mathbb L_p$ and $i^*\mathbb V^-=\mathbb J_1\oplus\cdots\oplus \mathbb J_q$. Accordingly, $\ch(\mathbb V^+)=\sum e^{c_1(\mathbb L_i)}$ and by the formula for the Adams operations on cohomology, we have $\psi^k(\ch\mathbb V^+)=\sum_i e^{kc_1(\mathbb L_i)}$ and $\psi^k(\ch\mathbb V^-)=\sum_j e^{kc_1(\mathbb J_j)}$. Thus, the two sides of Equation (\ref{eqn:adams-commute}) agree. 
\end{proof}

\begin{lem}[densely generating]\label{lem:densely-generate}
Let $H=\mathbb Q[[t_1,\ldots,t_r]]$ be a formal power series ring with variables $t_1,\ldots, t_r$. Write $d_i\ge1$ for the formal degree of $t_i$. If $a_i\in H$ has lowest-degree nonzero term $t_i$ for $i=1,\ldots,r$, then $1,a_1,\ldots,a_r$ generate a dense subalgebra of $H$. 
\end{lem}

\begin{proof}
A subset $H_0\subset H$ of a power series ring is dense if for any $b\in H$ and $n\ge0$, there exists $b_n\in H_0$ so that $b-b_n$ vanishes up to degree $n$. 

Let $H_0$ be the subalgebra generated by $1,a_1,\ldots,a_r$. Fixing $b\in H$, we define the desired $b_n$ inductively. For $n=0$, we take $b_0$ to be the constant term of $b$. For the induction step, the power series $b-b_{n-1}$ vanishes up to degree $n-1$, so the lowest-degree nonzero term of $b-b_{n-1}$ is a linear combination of monomials of the $t_I^{\alpha}=t_{i_1}^{\alpha_{i_1}}\cdots t_{i_m}^{\alpha_{i_m}}$ with total degree $\alpha_{i_1}d_{i_1}+\cdots+\alpha_{i_m}d_{i_m}=n$. Each such monomial is the lowest-degree nonzero term of $a_I^{\alpha}=a_{i_1}^{\alpha_{i_1}}\cdots a_{i_m}^{\alpha_{i_m}}$, and we define $b_n'$ to be the corresponding linear combination of the $a_I^{\alpha}$. By construction, for $b_n=b_{n-1}+b_n'$, the power series $b-b_n$ vanishes up to degree $n$. 
\end{proof}

\begin{lem}[spanning with Adams operations]\label{lem:adams-ops}
Let $H=\mathbb Q[[t_1,\ldots,t_r]]$ be a formal power series ring with $\deg(t_i)$ even for each $i$. Fix $c\in H$ and write $c=c_0+c_1+\cdots$, where $c_d$ is the degree $2d$ term of $c$. Let $H_0$ be the linear subspace of $H$ generated by the orbit of $c$ under Adams operations $\{c=\psi^1(c),\psi^2(c),\psi^3(c),\ldots\}$. Then for each $d$, the subspace $H_0$ contains an element whose lowest-degree nonzero term is $c_d$. 
\end{lem}

\begin{proof}
Let $\hat H_0$ be the subspace of $H$ spanned by $c_0, c_1, c_2,\ldots$.

By definition, $\psi^k(c)=c_0+kc_1+k^2c_2+k^3c_3\cdots$. Fix $d\ge1$ and consider the $(d+1)\times(d+1)$ matrix, whose $k$-th row gives the coefficients of $\psi^k(c)$ expressed as a linear combination of the $c_i$, up to degree $2d$. 
\[\left[
\begin{array}{ccccccc}
1&1&1^2& 1^3&\cdots & 1^d\\
1&2&2^2& 2^3&\cdots & 2^d\\
\vdots&\vdots&\vdots&\ddots & \vdots\\
1&d+1&(d+1)^2& (d+1)^3&\cdots & (d+1)^d\\
\end{array}
\right].\]
This matrix is invertible (Vandermonde determinant), so the standard basis vectors are in the row-span. Hence the span of $\{\psi^k(c)\}_{0\le k\le d}$ contains a power series whose lowest-degree nonzero term is $c_i$ for $i=0,\ldots,d$. Since $d$ was arbitrary, this completes the proof. 
\end{proof}


\begin{cor}\label{cor:dense}
Let $H=\mathbb Q[[t_1,\ldots,t_r]]$ be a formal power series ring with $\deg(t_i)$ even for each $i$. Assume that $\deg(t_i)\neq \deg(t_j)$ for $i\neq j$. Fix $c\in H$ and assume that the $\deg(t_i)$-term of $c$ contains a nonzero multiple of $t_i$ as a summand. Then the algebra spanned by $1$ and $\{\psi^k(c)\}_{k\ge0}$ is dense in $H$. 
\end{cor}

\begin{proof}
Assume without loss of generality that the $t_i$ are ordered by increasing degree. 
Write $H_0$ for the algebra spanned by 1 and $\{\psi^k(c)\}_{k\ge0}$. By Lemma \ref{lem:densely-generate}, it suffices to show that $H_0$ contains a power series $a_i$ whose lowest-degree nonzero term is $t_i$ for $i=1,\ldots,r$. We argue by induction on $i$. 

Write $c=\sum_{j\ge0} c_j$, where $c_j$ is the degree $2j$ term of $c$. For the base case, since $t_1$ is the monomial with smallest positive degree, we can take $a_1=c-c_0=\psi^1(c)-\psi^0(c)$. For the induction step, assume that we have found $a_1,\ldots,a_{i-1}$ such that the lowest-degree nonzero term of $a_j$ is $t_j$. By Lemma \ref{lem:adams-ops}, there is $a_i'$ in the span of the $\psi^k(c)$ whose lowest-degree nonzero term is $c_i$. By assumption we can write 
\[c_i=rt_i+\delta(t_1,\ldots,t_{i-1}),\]
where $r$ is a nonzero constant and $\delta\in\mathbb Q[t_1,\ldots,t_{i-1}]$ is a polynomial, each of whose monomials has (weighted) degree equal to $\deg(t_i)$. Observe that $\delta(a_1,\ldots,a_{i-1})$ is a power series whose lowest-degree nonzero term is equal to $\delta(t_1,\ldots,t_{i-1})$. Therefore $a_i:=\frac{1}{r}\big[c_i-\delta(a_1,\ldots,a_{i-1})\big]$ has lowest-degree nonzero term equal to $t_i$, as desired. 
\end{proof}

\begin{rmk}\label{rmk:constant}
For the purpose of applying the results of this subsection, we note that the image of $\gamma$ always contains $1\in H^{\ev}(BG)$, which is the image of the trivial representation $\mathbb C$. 
\end{rmk}

\subsection{\boldmath Dense image for $R(G)\otimes\mathbb Q\to H^{\ev}(BG)$}\label{sec:image}

In this section we prove Theorem \ref{thm:rational-surjectivity}. We begin with some notation. Recall that $K<G$ is a maximal compact subgroup. For representations, we use the shorthand $\sigma_k^\ell=\Sym^\ell(\mathbb C^k)$ and $\lambda_n^k=\Lambda^k(\mathbb C^n)$ and $1=\mathbb C$ (trivial representation); $V^*$ and $\overline V$ denote the dual and complex conjugate representations, respectively. When $n$ is fixed, we write $c_i\in \mathbb Q[x_1,\ldots,x_n]^{S_n}$ and $p_i\in\mathbb Q[x_1^2,\ldots,x_n^2]^{S_n}$ for  the degree-$i$ elementary symmetric polynomials in the given variables (these are Chern and Pontryagin classes). When $K=K_1\times K_2$ is a product (even up to isogeny) we will always use a ``diagonal" torus $T=T_1\times T_2\hookrightarrow K_1\times K_2$. For our computations we will (sometimes implicitly) use coordinates $H^*(BT_1)=\mathbb Q[x_1,\ldots,x_p]$ and $H^*(BT_2)=\mathbb Q[y_1,\ldots,y_q]$, unless otherwise specified. When $p$ or $q$ is 1, we drop the subscript. 

\begin{rmk}\label{rmk:isogeny}
When $K$ is the quotient of $K_1\times K_2$ by a finite group (as will occur below), we don't literally get a diagonal torus. But in practice this does not matter. 
An isogeny $T'\to T$ between tori induces an injection $H^1(T;\mathbb Z)\to H^1(T';\mathbb Z)$, and a basis for $H^1(T;\mathbb Z)$ gives one for $H^1(T';\mathbb Z)$ after rescaling, and this does not ultimately change our computation (of the closure of the image of $\gamma:R(G)\otimes\mathbb Q\to H^{\ev}(BG)$). 
\end{rmk}

\subsubsection*{\boldmath$\Sp_{2n}(\mathbb R)$ and $\SO^*(2n)$} 

We begin with 
\[\Sp_{2n}(\mathbb R)=\{A\in\SL_{2n}(\mathbb R): A^tJA=J\}\>\>\>\text{ where }J=\left(\begin{smallmatrix}
0 & I_n\\
-I_n & 0
\end{smallmatrix}\right).\]
This case was originally proved by Lusztig \cite{Lusztig}. We present his argument, in part for completeness, and also in order to contrast it with the argument in other cases. Notably Lusztig's argument does not use Adams operations. For many of the other cases, e.g.\ $\SU(p,q)$, trying to argue density without Adams operations seems to lead to a proof that is overly complicated. Similarly, using Adams operations for $\Sp_{2n}(\mathbb R)$ does not seem to be judicious.

\begin{proof}[Proof of Theorem \ref{thm:rational-surjectivity} for $\Sp_{2n}(\mathbb R)$]

Here $K=\U(n)$, so $H^{\ev}(BG)\cong\mathbb Q[[c_1,\ldots,c_n]]$. The standard representation $U=\mathbb C^{2n}$ is Hermitian. Indeed the action of $\Sp_{2n}(\mathbb R)$ on $\mathbb C^{2n}$ preserves the Hermitian form $iJ$ because $A\in\Sp_{2n}(\mathbb R)$ satisfies $A^tJA=J$ and $A^*=A^t$. As a representation of the maximal compact $K=\U(n)$, the representation $U$ decomposes $U=\lambda_n^1+(\lambda_n^1)^*$. Then 
\[\gamma(U)=\sum_{i=1}^n \left(e^{x_i}-e^{-x_i}\right).\]
Consequently, $\gamma(\Lambda^kU)$ is the degree-$k$ elementary symmetric polynomial in 
\[\{e^{x_1},-e^{-x_1},\ldots,e^{x_n},-e^{-x_n}\}.\]
Let $s_k$ be the degree-$k$ elementary symmetric polynomial in  $\{e^{x_1}-e^{-x_1},\ldots,e^{x_n}-e^{-x_n}\}$. 

The proof reduces to two observations (that are each easy to check): 
\begin{enumerate}
\item $s_k$ is a power series whose lowest-degree term is $2c_k$. 
\item $\gamma(\Lambda^kU)=s_k+\delta(s_0,\ldots,s_{k-1})$, where $\delta\in\mathbb Q[s_0,\ldots,s_{k-1}]$ is a polynomial. 
\end{enumerate} 
To conclude, (i) implies that subalgebra of $H^{\ev}(BG)$ generated by $s_0, s_1,\ldots, s_n$ is dense in $H^{\ev}(BG)$ (Lemma \ref{lem:densely-generate}), and (ii) implies that $s_k$ is in the subalgebra spanned by the $\gamma(\Lambda^iU)$ for $k=0,\ldots,n$ (proved inductively). 
\end{proof}

The proof of Theorem \ref{thm:rational-surjectivity} for 
\[\SO^*(2n)=\{
A\in \GL_{2n}(\mathbb C): A^*JA=J\text{ and }A^tA=I\},\>\>\>\text{ where }J= \left(\begin{smallmatrix}
0 & I_n\\
-I_n & 0
\end{smallmatrix}\right)\] follows from the argument for $\Sp_{2n}(\mathbb R)$. By definition $\SO^*(2n)$ is a subgroup of $\U(n,n)$, so the standard representation $U=\mathbb C^{2n}$ is Hermitian. Its restriction to $K=\U(n)$ decomposes $U=\lambda_n^1+(\lambda_n^1)^*$. This is the same computation as for $\Sp_{2n}(\mathbb R)$, and so the same argument applies here to show that $\gamma\otimes\mathbb Q$ has dense image for $G=\SO^*(2n)$. 

\subsubsection*{\boldmath$\SU(n_1,n_2)$ and $\Sp(n_1,n_2)$} 

The proofs are similar for 
\[\SU(n_1,n_2)=\{
A\in \SL_{n_1+n_2}(\mathbb C): A^*Q_{n_1,n_2}A=Q_{n_1,n_2}\},\]
where $Q_{n_1,n_2}= \left(\begin{smallmatrix}
I_{n_1} & 0\\
0 & -I_{n_2}
\end{smallmatrix}\right)$, and
\[\Sp(n_1,n_2)=\{A\in\GL_{2n}(\mathbb C): A^tJA=A\text{ and }A^*K_{p,q}A=K_{p,q}\},\] 
where $J= \left(\begin{smallmatrix}
0 & I_n\\
-I_n & 0
\end{smallmatrix}\right)$ and 
$K_{p,q}= \left(\begin{smallmatrix}
I_{n_1} & \\
 & -I_{n_2}\\
 &&I_{n_1}\\
 &&&-I_{n_2}
\end{smallmatrix}\right)$.

\begin{proof}[Proof of Theorem \ref{thm:rational-surjectivity} for $\SU(n_1,n_2)$]
  Here $K=S(\U(n_1)\times\U(n_2))$. Instead of using the standard coordinates for $H^{\ev}(BG)$, it will be convenient to use
  \[H^{\ev}(BG)\cong\mathbb Q[[p_1,\ldots,p_{n_1},q_2,\ldots,q_{n_2}]],\]
  where for $m\ge0$, we define 
\[p_m=\frac{1}{m!}(x_1^m+\cdots+x_{n_1}^m)\>\>\>\text{ and }\>\>\>q_m=\frac{1}{m!}(y_1^m+\cdots+y_{n_2}^m).\]

The standard representation of $\SU(n_1,n_2)$ on $\mathbb C^{n_1,n_2}$ is clearly Hermitian, and  
\[\gamma(U)=\sum_{i=1}^{n_1} e^{x_i}-\sum_{j=1}^{n_2}e^{y_j}=\sum_{m\ge0}(p_m-q_m).\] More generally, by Lemma \ref{lem:adams-commute}, we have \[\gamma(\psi^kU)=\sum_{i=1}^{n_1} e^{kx_i}+(-1)^k\sum_{j=1}^{n_2}e^{ky_j}=\sum_{m\ge0} k^m\big(p_m+(-1)^kq_m\big).\]
In particular, 
\[\gamma(\psi^2U)=\sum_{m\ge0} 2^m\big(p_m+q_m\big).\]

By Lemma \ref{lem:adams-ops} (applied twice), the image of $\gamma\otimes\mathbb Q$ contains elements whose lowest-degree nonzero term is $p_m+q_m$ and $p_m-q_m$ for each $m\ge0$. But then the same is true for $p_m$ and $q_m$ for $m\ge0$, and this implies that the image of $\gamma\otimes\mathbb Q$ is dense by Lemma \ref{lem:densely-generate}. 
\end{proof}

\begin{proof}[Proof of Theorem \ref{thm:rational-surjectivity} for $\Sp(n_1,n_2)$] 
Here $K=\Sp(n_1)\times\Sp(n_2)$, and we identify $H^{\ev}(BG)$ with $\mathbb Q[[p_1,\ldots,p_{n_1},q_1,\ldots,q_{n_2}]]$, where for $m\ge0$, we define 
\[p_m=\frac{2}{(2m)!}(x_1^{2m}+\cdots+x_{n_1}^{2m})\>\>\>\text{ and }\>\>\>q_m=\frac{2}{(2m)!}(y_1^{2m}+\cdots+y_{n_2}^{2m}).\] 
By definition, $\Sp(n_1,n_2)$ is a subgroup of $\U(2n_1,2n_2)$, so the standard representation $U=\mathbb C^{2(n_1+n_2)}$ is Hermitian, and 
\[
\gamma(U)=\sum_{i=1}^{n_1} \left(e^{x_i}+e^{-x_i}\right)-\sum_{j=1}^{n_2}\left(e^{y_j}+e^{-y_j}\right)=\sum_{m\ge0}p_m-q_m.
\]
By Lemma \ref{lem:adams-commute}, $\gamma(\psi^2U)=\sum_{m\ge0} 2^m\big(p_m+q_m\big)$. The formulas $\gamma(U)=\sum p_m-q_m$ and $\gamma(\psi^2U)=\sum 2^m(p_m+q_m)$ are the same as for the $\SU(n_1,n_2)$ case, and the rest of the proof that $\gamma\otimes\mathbb Q$ has dense image is the same. \end{proof}

\subsubsection*{\boldmath$\SO(n_1,n_2)$} 
For 
\[\SO(n_1,n_2)=\{
A\in \SL_{n_1+n_2}(\mathbb R): A^tQ_{n_1,n_2}A=Q_{n_1,n_2}\},\] 
where $Q_{n_1,n_2}= \left(\begin{smallmatrix}
I_{n_1} & 0\\
0 & -I_{n_2}
\end{smallmatrix}\right)$, 
the proof is similar to that for $\Sp(2m_1,2m_2)$, where $m_i=\lfloor n_i/2\rfloor$. Indeed the rings $H^{\ev}(BG)$ in these two cases differ by Euler classes, and we account for the Euler classes for $\SO(n_1,n_2)$ by using spin representations. We give a short explanation of the proof, which is followed by a discussion of spin representations, their Hermitian forms, and their Chern characters. The key results are Lemmas \ref{lem:spin-chern} and \ref{lem:even-spin-chern}. 

\begin{proof}[Proof of Theorem \ref{thm:rational-surjectivity} for $\SO(n_1,n_2)$]

By assumption $n_1$ and $n_2$ are not both odd (otherwise $G$ does not have discrete series). For concreteness, we assume that $n_1$ is even. 

Here $K\dot{=}\SO(n_1)\times \SO(n_2)$ and 
\[H^{\ev}(BG)\cong 
\begin{cases}
\mathbb Q[[p_1,\ldots,p_{m_1-1},e,q_1,\ldots,q_{m_2-1},f]]& n_2 \text{ even}\\
\mathbb Q[[p_1,\ldots,p_{m_1-1},e,q_1,\ldots,q_{m_2-1}]]& n_2 \text{ odd}
\end{cases}\] 
where $e=x_1\cdots x_{m_1}$, $f=y_1\cdots y_{m_2}$, and $p_i,q_i\in H^{4i}$ are Pontryagin classes in the $x,y$ variables, respectively. The standard representation $U=\mathbb R^{n_1,n_2}\otimes\mathbb C$ is Hermitian, and 
\[
\gamma(U)=\sum_{i=1}^{m_1} \left(e^{x_i}+e^{-x_i}\right)-\sum_{j=1}^{m_2}\left(e^{y_j}+e^{-y_j}\right).
\]
Given the similarity of this computation to that of $\Sp(2m_1,2m_2)$, we may use the argument there to reduce the proof to showing that each of the Euler classes $e,f$ appear as the lowest-degree nonzero term of some element of the image of $\gamma\otimes\mathbb Q$. For this we use the spin representations. See Lemmas \ref{lem:spin-chern} and \ref{lem:even-spin-chern} below. \end{proof}

The proof above relies on spin representations for (a cover of) $\SO(n_1,n_2)$. We denote these as 
\[\Spin(n_1,n_2) \to \Spin(n_1+n_2,\mathbb C) \to GL(S),\]
where $S$ is a complex representation of dimension $2^{m}$, where $m=m_1+m_2$ and $m_i=\lfloor \frac{n_i}{2}\rfloor$. (Here we've used that $n_1,n_2$ are not both odd -- in that case $\dim S=2^{m+1}$.) 

Our reference for the following discussion is Deligne \cite{deligne}. For concreteness we assume $n_1$ is even and consider cases when $n_2$ is odd or even. If $n_2$ is odd, then $S$ is an irreducible representation of $\Spin(n_1+n_2,\mathbb C)$, and if $n_2$ is even, then $S$ decomposes $S=S^e\oplus S^o$ (the so-called half-spin representations). The weights of $S$ are 
\begin{equation}\label{eqn:spin-weights}\frac{1}{2}(\pm L_1\pm\cdots\pm L_{m_1}\pm J_1\pm\cdots\pm J_{m_2}),\end{equation}
and the weights of the representation $S^e$ (resp.\ $S^o$) are those in (\ref{eqn:spin-weights}), where there are an even (resp.\ odd) number of ``$-$" signs (the superscripts $e$/$o$ are for even/odd).

\begin{lem}
The spin representation $S$ of $\Spin(2m_1,2m_2+1)$ is Hermitian, as are the half-spin representations $S^e,S^o$ of $\Spin(2m_1,2m_2)$. 
\end{lem}

\begin{proof}
First assume $n_2$ is odd. Recall that it suffices to show $S^*\cong \overline{S}$. First observe from the description of the weights that $S^*\cong S$ (the weights of $S^*$ are the negatives of the weights of $S$). Next we use that  $S$ has either a real or quaternionic structure as a representation of $\Spin(n_1,n_2)$; see \cite[\S1.4]{deligne}. This means that there is an antilinear equivariant map $S\to S$ whose square is $\pm1$, so this map gives an isomorphism $S\cong\overline S$. Thus $\overline S\cong S\cong S^*$, so $S$ is Hermitian. 

Next assume $n_2$ is even. By looking at the weights, we see that $S^e$ and $S^o$ are self-dual if $m_1+m_2$ is even, and $S^e$ and $S^o$ are dual to each other if $m_1+m_2$ is odd. 

If $m_1+m_2$ is even, then the signature $n_1-n_2$ is $0\pmod 4$, which implies that $S^e$ and $S^o$ have either an equivariant real or quaternionic structure \cite[loc.\ cit.]{deligne}, so again we conclude $\overline{S^e}\cong S^e\cong (S^e)^*$, and the same holds for $S^o$. 

If $m_1+m_2$ is odd, then $n_1-n_2\equiv 2\pmod 4$, so $S^e$ and $S^o$ are conjugate representations of $\Spin(n_1,n_2)$ \cite[loc.\ cit.]{deligne}. Thus $\overline{S^e}\cong S^o\cong (S^e)^*$ and $\overline{S^o}\cong S^e\cong (S^o)^*$. 
\end{proof}

Next we compute the $\gamma:R(G)\to H^{\ev}(BG)$ on the spin representations. 

For $\Spin(2m_1,2m_2+1)$, write $S_1$ and $S_2$ for the spin representations of $\Spin(2m_1)$ and $\Spin(2m_2+1)$, respectively. The representation $S_2$ is irreducible, and $S_1$ decomposes $S_1=S_1^e\oplus S_1^o$ with $S_1^e,S_1^o$ irreducible. Accordingly, the irreducible decomposition of $S\cong S_1\otimes S_2$ with respect to the action of the maximal compact $\Spin(2m_1)\times \Spin(2m_2+1)$ is given by 
\[S=(S_1^e\otimes S_2)\oplus (S_1^o\otimes S_2).\]
This also gives the Hermitian decomposition $S$, i.e.\ $S^+=S_1^e\otimes S_2$ and $S^-=S_1^o\otimes S_2$. Given this, we obtain the following expression for the $\gamma(S)=\ch(S^+)-\ch(S^-)$: 
\begin{equation}\label{eqn:spin-chern}
\gamma(S)=\sum_{\epsilon, \delta} \left(\prod_{i=1}^{m_1}\epsilon_i\right) e^{(\epsilon\cdot x+\delta\cdot y)/2}, 
\end{equation}
where $x=(x_1,\ldots,x_{m_1})$ and $\epsilon=(\epsilon_1,\ldots,\epsilon_{m_1})\in\{\pm1\}^{m_1}$ and $\epsilon\cdot x:=\sum \epsilon_ix_i$; similarly, $y=(y_1,\ldots,y_{m_2})$ and $\delta=(\delta_1,\ldots,\delta_{m_2})\in\{\pm1\}^{m_2}$; the sum ranges over all choices of $\epsilon$ and $\delta$. 

\begin{lem}\label{lem:spin-chern}
For the spin representation $S$ of $\Spin(2m_1,2m_2+1)$, the lowest-degree nonzero term of $\gamma(S)$ is a multiple of $e=x_1\cdots x_{m_1}$. 
\end{lem}

\begin{proof}
Fix $k\ge0$. We will show that the degree-$k$ term of $\gamma(S)$ is 0 for $k<m_1$. Using (\ref{eqn:spin-chern}), this term is (up to a scalar multiple) equal to 
\[
\sum_{\epsilon,\delta} \left(\prod_{i=1}^{m_1}\epsilon_i\right)\big(\epsilon_1x_1+\cdots+\epsilon_{m_1}x_{m_1}+\delta_1y_1+\cdots+\delta_{m_2}y_{m_2}\big)^k
\]
Fix a monomial 
\[x_I^{\alpha_I}y_J^{\beta_J}:=x_{i_1}^{\alpha_{i_1}}\cdots x_{i_r}^{\alpha_{i_r}}y_{j_1}^{\beta_{j_1}}\cdots y_s^{\beta_{j_s}}\]
with $r+s=k$. The coefficient on this monomial is (up to a scalar multiple, which is a multinomial coefficient) equal to 
\[\sum_{\epsilon,\delta}\left(\prod_{i=1}^{m_1}\epsilon_i\right)\epsilon_1^{\alpha_1}\cdots\epsilon_{m_1}^{\alpha_{m_2}}\delta_1^{\beta_1}\cdots\delta_{m_2}^{\beta_{m_2}}=\sum_{\epsilon,\delta}\left(\prod_{i=1}^{m_1}\epsilon_i^{\alpha_i+1}\right)\left(\prod_{j=1}^{m_2}\delta_j^{\beta_j}\right).\]
Here by convention $\alpha_i=0$ if $i\notin I=\{i_1,\ldots,i_r\}$ and $\beta_j=0$ if $j\notin J=\{j_1,\ldots,j_s\}$. 

If some index $i_0$ does not appear in $I=\{i_1,\ldots,i_r\}$, then since $\alpha_{i_0}=0$, then the terms in the sum with $\epsilon_{i_0}=\pm1$ cancel in pairs. Then in order for the coefficient on $x_I^{\alpha_I}y_J^{\beta_J}$ to be nonzero, we must have $I=\{1,\ldots,m_1\}$, so $k\ge m_1$. 

Finally, we show that the coefficient on $x_1\cdots x_{m_1}$ is nonzero. In terms of the preceding notation, this is the case $I=\{1,\ldots,m_1\}$ and $J$ is empty and $\alpha_i=1$ for each $i$. Then the coefficient is (up to a scalar multiple) equal to 
\[\sum_\epsilon\epsilon_i^2=2^{m_1}.\qedhere\]
\end{proof}

For $\Spin(2m_1,2m_2)$, write $S_1$ and $S_2$ for the spin representations of $\Spin(2m_1)$ and $\Spin(2m_2)$, respectively, and write $S_i=S_i^e\oplus S_i^o$ for the irreducible decomposition of these representations. Accordingly, the irreducible decomposition of $S^e$ and $S^o$ with respect to the action of  $K=\Spin(2m_1)\times \Spin(2m_2)$ is given as follows. 
\[S^e=(S_1^e\otimes S_2^e)\oplus (S_1^o\otimes S_2^o)\>\>\>\text{ and }\>\>\>
S^o=(S_1^e\otimes S_2^o)\oplus (S_1^o\otimes S_2^e).\]
Since each summand is irreducible, we conclude that the above decompositions of $S^e$ and $S^o$ agree with the Hermitian decomposition. Given this, we obtain, similar to the $\SO(2m_1,2m_2+1)$ case,  
\begin{equation}\label{eqn:even-spin-chern}
\gamma(S^e)=\sum_{\epsilon, \delta} \left(\prod_{i=1}^{m_1}\epsilon_i\right) e^{(\epsilon\cdot x+\delta\cdot y)/2}, 
\end{equation}
where the sum is over $\epsilon\in\{\pm1\}^{m_1}$ and $\delta\in\{\pm1\}^{m_2}$ such that $\bigl(\prod_{i=1}^{m_1}\epsilon_i\bigr)\bigl(\prod_{j=1}^{m_2}\delta_j\bigr)=1$, 
and 
\begin{equation}\label{eqn:odd-spin-chern}
\gamma(S^o)=\sum_{\epsilon, \delta} \left(\prod_{i=1}^{m_1}\epsilon_i\right) e^{(\epsilon\cdot x+\delta\cdot y)/2}, 
\end{equation}
where the sum is over $\epsilon,\delta$ such that $\bigl(\prod_{i=1}^{m_1}\epsilon_i\bigr)\bigl(\prod_{j=1}^{m_2}\delta_j\bigr)=-1$. 

\begin{lem}\label{lem:even-spin-chern}
Let $S^e,S^o$ be the half-spin representations of $\Spin(2m_1,2m_2)$. The lowest-degree nonzero term of $\gamma(S^e)+\gamma(S^o)$ is a multiple of $e=x_1\cdots x_{m_1}$. The lowest-degree nonzero term of $\gamma(S^e)-\gamma(S^o)$ is a multiple of $f=y_1\cdots y_{m_2}$. 
\end{lem}

\begin{proof}
The sum $\gamma(S^e)+\gamma(S^o)$ is equal to $\gamma(S)$, whose computation is the same as Lemma \ref{lem:spin-chern}. The computation of $\gamma(S^e)-\gamma(S^o)$ is similar, so we omit it. 
\end{proof}

\subsubsection*{\boldmath$G_{2(2)}$ }

The remaining sections treat the exceptional groups. A helpful reference is Adams \cite{adams}, which has many computations that we will use. 

\begin{proof}[Proof of Theorem \ref{thm:rational-surjectivity} for $G_{2(2)}$]

Here $K=[\SU(2)\times\SU(2)]/\mathbb Z_2$, so $H^{\ev}(BG)\cong\mathbb Q[[x^2,y^2]]$. The adjoint representation $A$, when restricted to $K$, decomposes 
\begin{align*}
A&=\sigma_2^3\otimes\lambda_2^1+\big[\mathfrak{su}(2)+\mathfrak{su}(2)\big]\\[3mm]
&=\sigma_2^3\otimes\lambda_2^1+\big[\sigma_2^2\otimes 1+1\otimes\sigma_2^2\big],\end{align*}
\cite[\S22]{FH}, and we compute 
\[\begin{array}{rl}
\gamma(A) &= (e^{3x}+e^x+e^{-x}+e^{-3x})(e^y+e^{-y})-(e^{2x}+e^{-2x}+e^{2y}+e^{-2y}+2)\\[2mm]
&=2+16x^2+(\text{higher order terms})
\end{array}
\]
We compute further
\[
\gamma(\Lambda^2A)=-5 - 16 x^2 - 16 y^2+(\text{h.o.t.})
\]
From this we quickly conclude that the image of $\gamma\otimes\mathbb Q$ is dense (by Lemma \ref{lem:densely-generate}). 
\end{proof}

\subsubsection*{\boldmath$F_{4(4)}$ }

Here $K=\big[\Sp(3)\times\SU(2)\big]/\mathbb Z_2$, and 
$H^{\ev}(BG)$ is isomorphic to $\mathbb Q[[p_1,p_2,p_3,y^2]]$. 

The adjoint representation $A$, when restricted to $K$, decomposes
\begin{align*}
A&=\hat\lambda_6^3\otimes\lambda_2^1+\big[\mathfrak{sp}(3)+\mathfrak{su}(2)\big]\\
&=\hat\lambda_6^3\otimes\lambda_2^1+\big[\sigma_6^2\otimes 1+1\otimes\sigma_2^2\big],
\end{align*}
where $\hat\lambda_6^3$ is the kernel of the contraction $\lambda_6^3\to\lambda_6^1$ given by evaluating the symplectic form. This computation can be done directly from the roots of $F_4$. We compute 
\begin{align*}
\gamma(A)&=\sum_{\epsilon\in\{\pm1\}^4}e^{\epsilon_1x_1+\epsilon_2x_2+\epsilon_3x_3+\epsilon_4y}+\sum_{1\le i\le 3}\>\sum_{\epsilon\in \{\pm1\}^2}e^{\epsilon_1x_i+\epsilon_2y}\\
&\quad{}-\left[4+(e^{2y}+e^{-2y})+\sum_{1\le i\le 3}(e^{2x_i}+e^{-2x_i})
+\sum_{1\le i<j\le 3}\>\sum_{\epsilon\in\{\pm1\}^2}e^{\epsilon_1x_i+\epsilon_2x_j}
\right]\\
&=4+ \big[2 p_1+ 10 y^2\big]+\left[ \frac{14}{3} p_2-\frac{5}{6} p_1^2+ 5 p_1 y^2-\frac{1}{6}y^4 \right]\\
&\quad{} +\left[\frac{23}{30} p_3+\frac{11}{15} p_1 p_2-\frac{29}{180} p_1^3+\frac{7}{6} p_2 y^2+ \frac{5}{12} p_1^2 y^2+\frac{5}{12}p_1 y^4-\frac{5}{36} y^6\right] +(\text{h.o.t.})
\end{align*}
and $\gamma(\Lambda^2A)=-18 -\big[28 p_1-4 y^2\big]+(\text{h.o.t.})$. By Lemma \ref{lem:adams-ops}, we conclude that each of 
\begin{gather*}
28 p_1-4 y^2\\
2 p_1+ 10 y^2\\
\frac{14}{3} p_2-\frac{5}{6} p_1^2+ 5 p_1 y^2-\frac{1}{6}y^4\\
\frac{23}{30} p_3+\frac{11}{15} p_1 p_2-\frac{29}{180} p_1^3+\frac{7}{6} p_2 y^2+ \frac{5}{12} p_1^2 y^2+\frac{5}{12}p_1 y^4-\frac{5}{36} y^6
\end{gather*}
is the lowest-degree nonzero term of some element in the image of $\gamma\otimes\mathbb Q$. Combining the first two, we obtain $p_1,y^2$ as lowest-degree nonzero terms, and then, in turn, we obtain $p_2$ and $p_3$ as lowest-degree nonzero terms. This proves the image of $\gamma\otimes\mathbb Q$ is dense by Lemma \ref{lem:densely-generate}.

\subsubsection*{\boldmath$F_{4(-20)}$ }

Here $K=\Spin(9)$, and $H^{\ev}(BG)\cong\mathbb Q[[p_1,p_2,p_3,p_4]]$. The adjoint representation $A$, when restricted to $K$, decomposes 
\[
A=S+\mathfrak{so}(9)=S+\lambda_9^2,
\]
where $S$ is the spin representation \cite[pg.\ 51]{adams}. We compute
\begin{align*}
\gamma(A)&=\sum _{\epsilon\in\{\pm1\}^4}e^{\frac{1}{2}(\epsilon_1x_1+\epsilon_2x_2+\epsilon_3x_3+\epsilon_4x_4)}-\sum_{i=1}^4\sum_{j=1}^4\sum_{\epsilon\in\{\pm1\}^2}e^{\epsilon_1x_i+\epsilon_2x_j}.\\
&=
-20-5 p_1+\frac{1}{3}p_2 -\frac{13}{24} p_1^2 \\
&\quad {}+\frac{5}{24} p_3 -\frac{1}{48}p_1 p_2 -\frac{11}{576} p_1^3\\
&\quad {}-\frac{221}{10080}p_4+\frac{79}{5040} p_1 p_3 -\frac{83}{40320} p_2^2 -\frac{109}{80640} p_1^2 p_2 -\frac{223}{645120} p_1^4+(\text{h.o.t.})
\end{align*}
Since the coefficient on each of $p_1,p_2,p_3,p_4$ is nonzero, we conclude using Corollary \ref{cor:dense} that the image of $\gamma\otimes\mathbb Q$ is dense.

\subsubsection*{\boldmath$E_{6(2)}$ }

Here $K\dot{=}\SU(6)\times\SU(2)$, and  $H^{\ev}(BG)\cong\mathbb Q[[c_2,\ldots,c_6,y^2]]$. The group $E_{6(2)}$ has a conjugate pair $U,\bar U$ of fundamental representations, $\dim U=27$. When restricted to $K\dot{=}\SU(6)\times\SU(2)$, 
\[U=\lambda_6^1\otimes\lambda_2^1+\lambda_6^4\otimes 1\>\>\>\text{ and }\>\>\>
\overline{U}=\lambda_6^5\otimes\lambda_2^1+\lambda_6^2\otimes 1.\]
See \cite[Table 15]{slanksy}. The representation $U$ of $G$ is Hermitian: from examining the decompositions for $U,\overline U$, we see that $\overline U\cong U^*$ (e.g.\ they have the same weights). Since $U$ splits into only two factors as a $K$-representation, this agrees with the decomposition $U=U^+\oplus U^-$. We compute 
\begin{align*}
\gamma(U)&=\sum_{1\le i\le 6}(e^{x_i+y}+e^{x_i-y})-\sum_{1\le i<j\le 6}e^{-x_i-x_j}\\
&=-3+(2 c_2+6 y^2)+2 c_3\\
&\quad+\left(-\frac{2}{3} c_4-\frac{1}{6}c_2^2-c_2 y^2+\frac{1}{2}y^4\right)
+\left(-\frac{1}{3}c_5+\frac{1}{2}c_3 y^2-\frac{1}{6}c_2 c_3\right)\\
&\quad {}-\frac{7}{30}c_6+\frac{1}{15}c_2 c_4-\frac{1}{120}c_3^2 +\frac{1}{180}c_2^3
-\frac{1}{12}c_2 y^4-\frac{1}{6}c_4 y^2+\frac{1}{12}c_2^2 y^2+\frac{1}{60}y^6\\
&\quad {}+(\text{h.o.t.})
\end{align*}
and $\gamma(\Lambda^2U)=-9+6 c_2-30 y^2+(\text{h.o.t.})$. By Lemma \ref{lem:adams-ops}, we conclude that each of 
\begin{gather*}
2 c_2+6 y^2\\
6 c_2-30 y^2\\
2 c_3\\
-\frac{2}{3} c_4-\frac{1}{6}c_2^2-c_2 y^2+\frac{1}{2}y^4\\
-\frac{1}{3}c_5+\frac{1}{2}c_3 y^2-\frac{1}{6}c_2 c_3\\
-\frac{7}{30}c_6+\frac{1}{15}c_2 c_4-\frac{1}{120}c_3^2 +\frac{1}{180}c_2^3
-\frac{1}{12}c_2 y^4-\frac{1}{6}c_4 y^2+\frac{1}{12}c_2^2 y^2+\frac{1}{60}y^6\\
\end{gather*}
is the lowest-degree nonzero term of some element in the image of $\gamma\otimes\mathbb Q$. Combining the first two, we obtain $c_2,y^2$ as lowest-degree nonzero terms, and then inductively we obtain $c_3,\ldots,c_6$. This proves the image of $\gamma\otimes\mathbb Q$ is dense by Lemma \ref{lem:densely-generate}.

\subsubsection*{\boldmath$E_{6(-14)}$ }

This is a case where it seems that the adjoint representation is insufficient to show that $\gamma\otimes\mathbb Q$ has dense image.

\begin{proof}[Proof of Theorem \ref{thm:rational-surjectivity} for $E_{6(-14)}$]

Here $K=\big[\Spin(10)\times\SO(2)\big]/\mathbb Z_4$, and 
\[H^{\ev}(BG)\cong\mathbb Q[[p_1,p_2,p_3,p_4,e,y]].\] The generators $y,p_1,p_2,e,p_3,p_4$ have dimension $2, 4, 8, 10, 12, 16$. 

The group $E_{6(-14)}$ has a conjugate pair $U,\bar U$ of fundamental representations, $\dim U=27$. When restricted to  $K=\big[\Spin(10)\times\SO(2)\big]/\mathbb Z_4$, 
\[U= 1\otimes\xi^{-4}+\lambda_{10}^1\otimes\xi^2+S^e\otimes\xi^{-1}\>\>\>\text{ and }\>\>\>
\overline{U}= 1\otimes\xi^{4}+\lambda_{10}^1\otimes\xi^{-2}+S^o\otimes\xi^{1},\]  
where $S^e$ is the ``even" half-spin representation, and $\xi\cong\mathbb C$ is the standard representation of $\SO(2)$ \cite[Cor.\ 8.3]{adams}. The representation $U$ of $G$ is Hermitian: from examining the decompositions for $U,\overline U$, we see that $\overline U\cong U^*$ (e.g.\ they have the same weights). The decomposition $U=U^+\oplus U^-$ is either $(\lambda_{10}^1\otimes\xi^2+1\otimes\xi^{-4})+S^e\otimes\xi^{-1}$
or 
$\lambda_{10}^1\otimes\xi^2+(1\otimes\xi^{-4}+S^e\otimes\xi^{-1})$. Other possibilities are ruled out using the fact that the real rank of $E_{6(-14)}$ is 2. Given this, we compute 
\[\gamma(U)=
\pm e^{-4y}+\sum_{1\le i\le 5}\>\sum_{\epsilon=\pm1} e^{\epsilon x_i+2y}-\sum_{\substack{\epsilon\in\{\pm1\}^5\\ \prod\epsilon_i=1}}e^{\frac{1}{2}(\epsilon_1x_1+\cdots+\epsilon_5x_5)-y}.\]
The sign ambiguity on the summand $e^{-4y}$ does not ultimately effect our computation. Write $\gamma(U)=u_0+u_2+u_4+u_6+\cdots$, where $u_{2i}$ is a homogeneous polynomial of (cohomological) degree $2i$. Then 
\begin{align*}
u_2&=-36 y \pm 4 y\\
u_4&=p_1-(12\pm 8) y^2 \\
u_8&= \frac{1}{3}p_2 -\frac{1}{24}p_1^2-p_1 y^2-\left(6\pm \frac{32}{3} \right)y^4 \\
u_{10}&=\frac{1}{2}e
-\frac{5}{24}p_1^2 y+\frac{1}{6}p_2 y-\frac{5}{3} p_1 y^3+\left(\pm \frac{128}{15} -\frac{14}{5}\right)y^5 \\
u_{12}&=\frac{1}{120}p_3+\frac{1}{80}p_1 p_2-\frac{7}{2880} p_1^3  -\frac{7}{48} p_1^2 y^2 +\frac{5}{12}p_2 y^2 -\frac{7}{12}p_1 y^4 - \frac{1}{2}ye-\left(\frac{13}{15} \pm \frac{256}{45}\right) y^6 \\
u_{16}&=\frac{19}{10080} p_4+\frac{1}{5040}p_1 p_3
-\frac{1}{13440}p_2^2+\frac{19}{80640}p_1^2 p_2-\frac{31}{645120} p_1^4-\frac{31}{5760}p_1^3 y^2+\frac{3}{160} p_1 p_2 y^2\\
    & \quad {}-\frac{1}{120}p_3 y^2-\frac{31}{360}p_1 y^6 +\frac{17}{144}p_2 y^4-\frac{31}{576} p_1^2 y^4 - \frac{1}{48}p_1 ye- \frac{1}{12}y^3e-\left(\frac{53}{840}\pm \frac{512}{315}\right) y^8.
\end{align*}
The key point is that since the coefficients on each of $y,p_1,p_2,e,p_3,p_4$ are nonzero, Corollary \ref{cor:dense} implies that the image of $\gamma\otimes\mathbb Q$ is dense. 
\end{proof}

\subsubsection*{\boldmath$E_{7(7)}$ }

Here $K=\SU(8)$, so $H^{\ev}(BG)\cong \mathbb Q[[c_2,\ldots,c_8]]$. The fundamental representation $U$ of $E_{7(7)}$ is 56-dimensional and when restricted to  $K=\SU(8)$, it decomposes
\[U=\lambda_8^2+\lambda_8^6.\]
See \cite[Thm.\ 11.1]{adams}. The representation $U$ of $G$ is Hermitian: $\overline{U}\cong U\cong U^*$ because $E_{7}$ has a unique 56-dimensional representation. Then 
\begin{align*}
\gamma(U)&=\sum_{1\le i<j\le 8} \big(e^{x_i+x_j}-e^{-(x_i+x_j)}\big)\\[3mm]
&=
4 c_3+\left(-\frac{2}{3} c_5-\frac{1}{3}c_2 c_3\right)
+\left(-\frac{7}{45}c_7+\frac{13}{180} c_2 c_5-\frac{1}{90}c_3 c_4+\frac{1}{90}c_2^2 c_3\right)+(\text{h.o.t.})
\end{align*}
We will also use the adjoint representation $A$. When restricted to $K$, there is a decomposition 
\[A=\lambda_8^4+\mathfrak{su}(8),\]
and we compute
\begin{align*}
\gamma(A)&=\sum_{1\le i_1<i_2<i_3<i_4\le 8}e^{x_{i_1}+x_{i_2}+x_{i_3}+x_{i_4}}-\left[7+\sum_{1\le i<j\le 8}(e^{x_i-x_j}+e^{x_j-x_i})
\right]\\
&=7
-4 c_2
+\left(\frac{16}{3} c_4-\frac{2}{3} c_2^2\right)
+\left(-\frac{8}{15} c_6-\frac{4}{5} c_2 c_4+\frac{4}{15} c_3^2+\frac{7}{45} c_2^3\right)\\
& \quad {}+\left(\frac{152}{315}c_8+\frac{16}{315} c_2 c_6+\frac{11}{630} c_3 c_5 -\frac{2}{105}c_4^2+\frac{19}{315} c_2^2 c_4-\frac{43}{1260} c_2 c_3^2-\frac{31}{2520} c_2^4\right).
\end{align*}
Observe that the coefficients of $c_3,c_5,c_7$ in $\gamma(U)$ are nonzero, and the coefficients of $c_2,c_4,c_6,c_8$ in $\gamma(A)$ are nonzero. Then using Corollary \ref{cor:dense} (twice -- separately for the even/odd-index variables), we conclude that the image of $\gamma\otimes\mathbb Q$ is dense.

\subsubsection*{\boldmath$E_{7(-5)}$ }

Here $K=\big[\Spin(12)\times\SU(2)\big]/\mathbb Z_2$, and so $H^{\ev}(BG)$ is isomorphic to $\mathbb Q[[p_1,\ldots,p_5,e,y^2]]$. The generators $p_1,y^2,p_2,p_3,e,p_4,p_5$ have degree $4,4,8,12,12,16,20$, respectively. 

The fundamental representation $U$ of $E_{7(-5)}$ is 56-dimensional and when restricted to  $K=\big[\Spin(12)\times\SU(2)\big]/\mathbb Z_2$, 
\[U=\lambda_{12}^1\otimes\lambda_2^1+S^o\otimes 1,\]
where $S^o$ is the ``odd" half-spin representation \cite[Cor.\ 8.2]{adams}. The representation $U$ of $G$ is Hermitian: $\overline{U}\cong U\cong U^*$ because $E_{7}$ has a unique 56-dimensional representation. Then 
\[
\gamma(U)=\sum_{1\le i\le 6}\>\sum_{\epsilon\in\{\pm1\}^2} e^{\epsilon_1 x_i+\epsilon_2y}- \sum_{\substack{\epsilon\in\{\pm1\}^6\\ \prod\epsilon_i=-1}}e^{\frac{1}{2}(\epsilon_1x_1+\cdots+\epsilon_6x_6)}
\]
Write $\gamma(U)=u_0+u_4+u_8+u_{12}+\cdots$ where $u_{4i}$ is a homogeneous polynomial of (cohomological) degree $4i$. Similarly write $\gamma(\Lambda^2U)=\sum v_{4i}$. We compute  
\begin{align*}
u_4&=4 p_1 -24 y^2\\
u_8&=16 p_2-2 p_1^2  -24 p_1 y^2-24 y^4\\
u_{12}&=12 p_3-360 e-\frac{7}{2}p_1^3+ 18 p_1 p_2  -60 p_1^2 y^2+  120 p_2 y^2 -60 p_1 y^4 -24 y^6\\
u_{16}&=152 p_4-\frac{31}{8}p_1^4+ 19 p_1^2 p_2 -6 p_2^2+ 16 p_1 p_3  -112 p_1^3 y^2+ 336 p_1 p_2 y^2\\
& \quad {}-336 p_3 y^2 -280 p_1^2 y^4+ 560 p_2 y^4 -112 p_1 y^6-24 y^8-840 p_1e\\
u_{20}&=1220 p_5-\frac{127}{32}p_1^5+ \frac{85}{4}p_1^3 p_2 -\frac{35}{2}p_1 p_2^2+ 5 p_1^2 p_3+ 40 p_2 p_3+ 310 p_1 p_4  -180 p_1^4 y^2\\
&\quad {}+ 720 p_1^2 p_2 y -360 p_2^2 y^2 -720 p_1 p_3 y^2+ 720 p_4 y^2
                                                                                                                                                  -840 p_1^3 y^4+ 2520 p_1 p_2 y^4\\
    & \quad {}-2520 p_3 y^4 
-840 p_1^2 y^6+ 1680 p_2 y^6 -180 p_1 y^8-24 y^{10}-945 p_1^2e -1260 p_2e\\
v_4&=
8 p_1 -240 y^2\\
v_8&=128 p_2-16 p_1^2+ 1344 y^4 -960 p_1 y^2\\
v_{12}&=-1056 p_3+ 8640e-292 p_1^3+ 1296 p_1 p_2 + 7680 y^6 -3120 p_1^2 y^2 -960 p_2 y^2+ 
 4800 p_1 y^4.
 \end{align*}
To conclude that $\gamma\otimes\mathbb Q$ has dense image, we use Lemma \ref{lem:adams-ops} to deduce that each of $u_{4i}, v_{4i}$ is the lowest-degree nonzero term of some power series $\hat u_{4i},\hat v_{4i}$ in the image of $\gamma\otimes\mathbb Q$. Combining $\hat u_4,\hat w_4$, we reach the same conclusion for $p_1,y^2$. Using this and $\hat u_8$, we reach the same conclusion for $p_2$. Using $p_1,y^2,p_2$ and $\hat u_{12}$ and $\hat v_{12}$ we obtain $12 p_3-360 e$ and hence $-1056 p_3+ 8640e$, and hence also $p_3$ and $e$ (as lowest degree nonzero terms). Finally, in a similar fashion, we inductively obtain $p_4, p_5$. Altogether, this proves the image of $\gamma\otimes\mathbb Q$ is dense by Lemma \ref{lem:densely-generate}. 
\qed

\begin{rmk}
For the computation of $E_{8(-24)}$, we will also need to know the weights of the adjoint representation of $G=E_{7(-5)}$ restricted to $K=\big[\Spin(12)\times\SU(2)\big]/\mathbb Z_2$, so we record that here. The adjoint representation $A$ decomposes as a $K$-representation 
\[
A=S^e\otimes\lambda_2^1+\big(\mathfrak{so}(12)+\mathfrak{su}(2)\big)=S^e\otimes\lambda_2^1+\big(\lambda_{12}^2\otimes 1+1\otimes \sigma_2^2\big).\]
See \cite[pg.\ 52]{adams}. 
The weights of $S^e\otimes\lambda_2^1$ are $\frac{1}{2}(\pm L_1\cdots\pm L_6)\pm J$, where there are an even number of ``$-$" on the summands $L_1,\ldots,L_6$. The weights of $\big(\lambda_{12}^2+\sigma_2^2\big)$ are $\pm L_i\pm L_j$ for $1\le i<j\le 6$ and $\pm 2J$, and $0$ with multiplicity 7. 
\end{rmk}

\subsubsection*{\boldmath$E_{7(-25)}$ }

Here $K=E_6\times\SO(2)$, and $H^{\ev}(BG)$ is isomorphic to 
\[\mathbb Q[[I_2,I_5,I_6,I_8,I_9,I_{12},z]],\] where $I_k$ is a polynomial in $x_1,\ldots,x_5,y$ of degree $k$ (and hence belongs to an element of $H^{2k}(BG)$).  In particular, $z$, $I_2$, $I_5$, $I_6$, $I_8$, $I_9$, $I_{12}$ have cohomological degrees $2$, $4$, $10$, $12$, $16$, $18$, $24$, respectively. 

The fundamental representation $U$ of $E_{7(-25)}$ is 56-dimensional and when restricted to  $K=E_6\times\SO(2)$, 
\[U=1\otimes\xi^3+\overline{W}\otimes\xi+W\otimes\xi^{-1}+1\otimes\xi^{-3},\]
where $W,\bar W$ are the fundamental representations of $E_6$ \cite[Table 15]{slanksy}. The representation $U$ of $G$ is Hermitian: : $\overline{U}\cong U\cong U^*$ because $E_{7}$ has a unique 56-dimensional representation. Consequently, the decomposition $U=U^+\oplus U^-$ satisfies $\dim U^+=\dim U^-$. This almost determines the decomposition -- we have not determined whether the decomposition pairs $1\otimes \xi^3$ with $\overline{W}\otimes\xi$ or with $W\otimes\xi^{-1}$. In any case, we are able to prove Theorem \ref{thm:rational-surjectivity} for $G=E_{7(-25)}$ without resolving this. 

To compute $\gamma(U)$, we ues the decomposition 
\[W\cong 1\otimes\eta^{-4}+\lambda_{10}^1\otimes\eta^2+S^e\otimes\eta^{-1}\] of $W$ as a representation of $[\Spin(10)\times\SO(2)]/\mathbb Z_4<E_6$ (that appeared in the discussion of $E_{6(-14)}$); here we write $\eta=\mathbb C$ for the standard representation of the $\SO(2)$ factor, and use $\eta$ instead of $\xi$ as above to distinguish the two factors of $\SO(2)$ (one appearing in $[\Spin(10)\times\SO(2)]/\mathbb Z_4\hookrightarrow E_6$ and other appearing in $K=E_6\times\SO(2)$). Similarly, 
\[\overline W\cong 1\otimes\eta^{4}+\lambda_{10}^1\otimes\eta^{-2}+S^o\otimes\eta^{1}.\]

We have 
\begin{align*}
\gamma(U)&=\pm(e^{3z}-e^{-3z})\\
&\quad {} +e^{-4y-z}+\sum_{1\le i\le 5}\>\sum_{\epsilon=\pm1} e^{\epsilon x_i+2y-z}+\sum_{\substack{\epsilon\in\{\pm1\}^5\\ \prod\epsilon_i=1}}e^{\frac{1}{2}(\epsilon_1x_1+\cdots+\epsilon_5x_5)-y-z}\\[3mm]
&\quad {} -\Bigl[e^{4y+z}+\sum_{1\le i\le 5}\>\sum_{\epsilon=\pm1} e^{\epsilon x_i-2y+z}+\sum_{\substack{\epsilon\in\{\pm1\}^5\\ \prod\epsilon_i=-1}}e^{\frac{1}{2}(\epsilon_1x_1+\cdots+\epsilon_5x_5)+y+z}\Bigr].
\end{align*}
Write $\gamma(U)=\sum u_i$ and $\gamma(\Lambda^2U)=\sum v_i$, where $u_i,v_i$ are homogeneous polynomials of degree $i$. 
We compute 
\begin{align*}
u_2&= 54 z\pm 6 z\\
u_4&= 0\\
u_{10}&= -2I_5+\frac{5}{6}I_2^2 z+20 I_2z^3 +(540\pm 486) z^5\\
u_{12}&= 0\\
u_{16}&= 0\\
u_{18}&=  -2 I_9+ 18 I_8z +168 I_6z^3-21I_5I_2 z^2- 252  I_5z^4+ 21 I_2^2 z^5 + 72  I_2 z^7\\
&\quad {} +(39420\pm39366) z^9\\
u_{24} &= 0,
\end{align*}
and working modulo the ideal generated by $z$, we compute  
\begin{align*}
v_2&\equiv  0\\
v_4&\equiv - 4 I_2\\
v_{10}&\equiv 0\\
v_{12}&\equiv - 64 I_6\\
v_{16}&\equiv  - 256 I_8\\
v_{18}&\equiv 0\\
v_{24}&\equiv  - 4096 I_{12}+924 I_5^2I_2.
\end{align*}
Now we conclude that $\gamma\otimes\mathbb Q$ has dense image after applications of Lemmas \ref{lem:densely-generate} and \ref{lem:adams-ops}.

\subsubsection*{\boldmath$E_{8(8)}$ }

Here $K=\Spin(16)/\mathbb Z_2$, so $H^{\ev}(BG)$ is isomorphic to $\mathbb Q[[p_1,\ldots,p_7,e]]$. The fundamental representation of $E_{8(8)}$ is the adjoint representation $A$. When restricted to $K=\Spin(16)/\mathbb Z_2$, 
\[A\cong S^e+\mathfrak{so}(16).\]
See \cite[Ch.\ 6]{adams}. Then 
\[
\gamma(A)=\sum_{\substack{\epsilon\in\{\pm1\}^8\\ \prod\epsilon_i=1}}e^{\frac{1}{2}(\epsilon_1x_1+\cdots+\epsilon_8x_8)}- \sum_{1\le i<j\le 8}\>\sum_{\epsilon\in\{\pm1\}^2}
e^{\epsilon_1 x_i+\epsilon_2x_j}-8.\]
For ease of notation, we write $\gamma(A)=a_0+\frac{1}{4!}a_4+\frac{1}{8!}a_8+\frac{1}{12!}a_{12}+\cdots$, where $a_{4i}$ is a homogeneous polynomial of (cohomological) degree $4i$. We compute 
\begin{align*}
a_4&=4 p_1\\
a_8&={64 p_2 -20 p_1^2}\\
a_{12}&={192 p_3+ 48 p_1 p_2 -26 p_1^3}\\
a_{16}&={20160 e -352 p_4+ 688 p_1 p_3 -104 p_2^2+ 
 12 p_1^2 p_2 -\frac{55}{2}p_1^4}\\
a_{20}&={9920 p_5 -3080 p_1 p_4+ 160 p_2 p_3+ 
 1820 p_1^2 p_3 -430 p_1 p_2^2 -35 p_1^3 p_2 -\frac{223}{8} p_1^5+75600 p_1e}\displaybreak\\
a_{24}&=41952 p_6+ 38568 p_1 p_5+ 9528 p_2 p_4 -13818 p_1^2 p_4 -4344 p_3^2+ 
 1800 p_1 p_2 p_3+ 
 4116 p_1^3 p_3 \\
& \quad {}-206 p_2^3 -\frac{2337}{2} p_1^2 p_2^2 -\frac{753}{8} p_1^4 p_2 -\frac{895}{32} p_1^6+166320 p_2e+ 124740 p_1^2e\\
a_{28}&=1337728 p_7+ 173600 p_1 p_6 -56448 p_2 p_5+ 116690 p_1^2 p_5+ 31640 p_3 p_4\\
&\quad {}+  45010 p_1 p_2 p_4 -\frac{86233}{2}p_1^3p_4 -30548 p_1 p_3^2+ 2394 p_2^2 p_3\\
&\quad {}+ 9471 p_1^2 p_2 p_3+ \frac{65653}{8}p_1^4 p_3-\frac{3241}{2}p_1 p_2^3 -\frac{20573}{8} p_1^3 p_2^2 -\frac{5355}{32}p_1^5 p_2 -\frac{3583}{128} p_1^7\\
&\quad {}+720720 p_3e+ 540540 p_1 p_2e+ 135135 p_1^3e.
\end{align*}
Similarly writing $\Lambda^2(A)=\sum \frac{1}{(4i)!}b_{4i}$, we have 
\begin{align*}
b_4&={-88 p_1}\\
b_8&={512 p_2 -400 p_1^2}\\
b_{12}&={1536 p_3+ 4224 p_1 p_2 -2368 p_1^3}\\
b_{16}&={-2419200 e -187136 p_4+ 82304 p_1 p_3+ 127168 p_2^2 -85664 p_1^2 p_2+ 7220 p_1^4}.
\end{align*}
To conclude that $\gamma\otimes\mathbb Q$ has dense image, we use Lemma \ref{lem:adams-ops} to deduce that each of $a_{4i}$ for $1\le i\le 7$ is the lowest-degree nonzero term of some element in the image of $\gamma\otimes\mathbb Q$. Then inductively we reach the same conclusion for $p_1,p_2,p_3$. Next we use $a_{16}$ and $b_{16}$ to show the same claim for $p_4$ and $e$. Then we continue inductively to obtain the same claim for $p_5,p_6,p_7$. Altogether, this proves the image of $\gamma\otimes\mathbb Q$ is dense by Lemma \ref{lem:densely-generate}.

\subsubsection*{\boldmath$E_{8(-24)}$ }

Here $K=E_7\times\SU(2)$, so $H^{\ev}(BG)$ is isomorphic to 
\[\mathbb Q[I_2,I_6,I_8,I_{10},I_{12},I_{14},I_{18},z^2].\] The generators 
$z^2,I_2,I_6,I_8,I_{10},I_{12},I_{14},I_{18}$ have cohomological degrees $4$, $4$, $12$, $16$, $20$, $24$, $28$, $36$.

The fundamental representation of $E_{8(-24)}$ is the adjoint representation $A$. When restricted to $K=E_7\times\SU(2)$, 
\[
A=U\otimes\lambda_2^1+(\mathfrak e_7+\mathfrak{su}(2)),
\]
where $U$ is the fundamental representation of $E_7$ \cite[pg.\ 54]{adams}. For computation, we use that $\big[\Spin(12)\times\SU(2)\big]/\mathbb Z_2\hookrightarrow E_7$ (a maximal rank subgroup, which appears in the computation for $E_{7(-5)}$), and with respect to this subgroup, we have the following decompositions 
\[U\cong\lambda_{12}^1\otimes\lambda_2^1+S^o\otimes 1\>\>\>\text{ and } \>\>\>
\mathfrak e_7\cong S^e\otimes\lambda_2^1+\big(\lambda_{12}^2\otimes 1+1\otimes \sigma_2^2\big).
\]

Then computing weights of $\big[\Spin(12)\times\SU(2)\big]/\mathbb Z_2\times\SU(2)$, we have
\begin{align*}
\gamma(A) &=\sum_{1\le i\le 6}\>\sum_{\epsilon\in\{\pm1\}^3} e^{\epsilon_1 x_i+\epsilon_2y+\epsilon_3z}+ \sum_{\substack{\epsilon\in\{\pm1\}^6\\ \prod_{i=1}^6\epsilon_i=-1}}e^{\frac{1}{2}(\epsilon_1x_1+\cdots+\epsilon_6x_6)+\epsilon_7z}
\\[3mm]
& \quad{} - \left[8+e^{2y}+e^{-2y}+e^{2z}+e^{-2z}\right. \\
& \qquad\quad\left.+\sum_{\substack{\epsilon\in\{\pm1\}^6\prod_{i=1}^6\epsilon_i=1}} e^{\frac{1}{2}(\epsilon_1 x_1\cdots\epsilon_6 x_6)+\epsilon_7 y}+\sum_{1\le i<j\le 6}\>\sum_{\epsilon\in\{\pm1\}^2}
e^{\epsilon_1 x_i+\epsilon_2x_j}\right].
\end{align*}
Write $\gamma(A)=\sum a_i$, where $u_i$ is a homogeneous polynomial of (cohomological) degree $i$. We compute 
\begin{align*}
a_4&=- I_2+104 z^2 \\
a_{12}&= 4 I_6- \frac{5}{288}I_2^3+ \frac{5}{4} z^2 I_2^2 + 30 z^4 I_2 -16 z^6\\
a_{16}&=- 8 I_8+ \frac{7}{3} I_2 I_6 - \frac{35}{6912}I_2^4 + 56 z^2 I_6 + \frac{35}{6} z^4 I_2^2 + 56 z^6 I_2 -400 z^8 \\
a_{20}&=4 I_{10}+(\text{a polynomial in $z^2,I_2,I_6,I_8$})\\
a_{24}&=32 I_{12}+(\text{a polynomial in $z^2,I_2,I_6,I_8,I_{10}$})\\
a_{28}&=\frac{484}{29}I_{14}+(\text{a polynomial in $z^2,I_2,I_6,I_8,I_{10},I_{12}$})\\
a_{36}&=\frac{114116}{1229}I_{18}+(\text{a polynomial in $z^2,I_2,I_6,I_8,I_{10},I_{12},I_{14}$}).
\end{align*}
And for $\gamma(\Lambda^2A)=\sum b_i$, we have 
\[b_4=14 I_2 -2736 z^2.\]
By Lemma \ref{lem:adams-ops}, we conclude that each of $a_4,b_4,a_{12},a_{16},a_{20},a_{24},a_{36}$ is the lowest-degree nonzero term of some element in the image of $\gamma\otimes\mathbb Q$. Combining the corresponding power series for $a_4,b_4$, we obtain $I_2,z^2$ as lowest-degree nonzero terms, and then inductively we obtain $I_6,I_8,I_{10},I_{12},I_{14},I_{18}$ as lowest-degree nonzero terms. This proves the image of $\gamma\otimes\mathbb Q$ is dense by Lemma \ref{lem:densely-generate}.

\bibliographystyle{amsalpha}
\bibliography{pl}
\end{document}